\documentclass[11.5pt,twoside,openany]{article}
\usepackage{amssymb,amsthm}
\usepackage[tbtags]{amsmath}
\usepackage{cite}
\usepackage{indentfirst}
\usepackage{cases}
\usepackage{graphicx}
\usepackage{subfigure}
\usepackage{lineno}
\usepackage{enumerate}
\usepackage{threeparttable}
\usepackage{multirow}
\usepackage{booktabs}
\usepackage{upgreek}  
\usepackage{bm}
\usepackage{fancyhdr}
\usepackage{ifthen}
\usepackage{lipsum}
\allowdisplaybreaks[4]
\usepackage[misc]{ifsym}
\usepackage{caption}
\usepackage{float}
\usepackage{appendix}
\usepackage[format=hang]{caption}
\usepackage{algorithm,algorithmic}

\usepackage{charter}  
\usepackage{indentfirst}  
\usepackage{soul}
\usepackage{authblk}

\newcommand{\rmnum}[1]{\romannumeral #1} 
\newcommand{\Rmnum}[1]{\uppercase\expandafter{\romannumeral #1}} 

\usepackage[colorlinks,
linkcolor=blue,
anchorcolor=blue,
citecolor=red]{hyperref}

\numberwithin{equation}{section}

\newtheorem{Lemma}{Lemma}[section]
\newtheorem{Theorem}{Theorem}[section]
\newtheorem{Remark}{Remark}[section]

\newcounter{saveeqn}

\makeatletter
\def\@maketitle{%
	\newpage
	\null
	\vskip 2em%
	\begin{center}%
		\let \footnote \thanks
		{\LARGE \@title \par}%
		\vskip 1.5em%
		{\large
			\lineskip .5em%
			\begin{tabular}[t]{c}%
				\@author
			\end{tabular}\par}%
	\end{center}%
	\par
	\vskip 1.5em}
\makeatother

\makeatletter
\evensidemargin
\oddsidemargin
\makeatother
\title{\bf  Second-order unconditionally stable time-filtered scheme for Cahn-Hilliard-Navier-Stokes system}
\author{
	Xi Li\footnote{School of Mathematical Sciences, Chengdu University of Technology, Chengdu, Sichuan 610059, China (lixi@cdut.edu.cn). The work of this author is supported by the Natural Science Foundation of Sichuan Province(No. 2025ZNSFSC0070).}, 
	Chun Song\footnote{School of Mathematics, Sichuan University, Chengdu, Sichuan 610064, China (song\_chun@stu.scu.edu.cn).}, 
	Haijun Gao\footnote{School of Mathematics, Shandong University, Jinan, Shandong, 250100, China (gaohaijun@sdu.edu.cn).},
	\ and Minfu Feng\footnote{School of Mathematics, Sichuan University, Chengdu, Sichuan 610064, China (fmf@scu.edu.cn). The work of this author was supported by the National Natural Science Foundation of China (Grant No. 11971337).}
}

\begin{document}
	\maketitle
	\newcommand\blfootnote[1]{%
		\begingroup
		\renewcommand\thefootnote{}\footnote{#1}%
		\par\setlength\parindent{2em}
		\endgroup
	}
	\captionsetup[figure]{labelfont={bf},labelformat={default},labelsep=period,name={Fig.}}
	\captionsetup[table]{labelfont={bf},labelformat={default},labelsep=period,name={Tab.}}
	\begin{abstract}
		In this work, we introduce the time filtering technique to develop several innovative semi-discrete schemes in time for the Cahn–Hilliard–Navier–Stokes (CHNS) system. These schemes achieve second-order temporal accuracy while maintaining unconditional energy stability. Our approach begins with the discretization of the CHNS system using the first-order semi-implicit method. Subsequently, by applying time filtering techniques, we improve the temporal accuracy from first-order to second-order. This improvement requires only minor modifications to the original first-order semi-implicit scheme, thereby enabling higher accuracy to be achieved at minimal cost. Moreover, we rigorously establish the unconditional energy stability of the proposed schemes through theoretical analysis. Additionally, we extend our work to develop semi-discrete schemes that incorporate variable and adaptive time-stepping strategies, enhancing the flexibility and efficiency of simulations. Numerical examples are presented to validate the theoretical results and demonstrate the effectiveness of the proposed methods.\\
		\\
		
	\end{abstract}
	\baselineskip 15pt
	\parskip 10pt
	\setcounter{page}{1}
	\vspace{-0.5cm}
	\section{Introduction}
	In this work, we shall consider the second-order in time, unconditional energy-stable time-stepping, and low computational complexity numerical approximation for the following matched-density Cahn-Hilliard-Navier-Stokes (CHNS) model \cite{2017_Diegel_Convergence_analysis_and_error_estimates_for_a_second_order_accurate_finite_element_method_for_the_Cahn_Hilliard_Navier_Stokes_system, 2023_CaiWentao_Optimal_L2_error_estimates_of_unconditionally_stable_finite_element_schemes_for_the_Cahn_Hilliard_Navier_Stokes_system}:
	\begin{subequations}\label{eq_chns_equations}
		\begin{align}
			\label{eq_phi_con}
			\frac{\partial \phi }{\partial t}+(\boldsymbol{u}\cdot\nabla)\phi-  \epsilon \nabla \cdot (M(\phi) \nabla \mu)=0,\quad \text{in} \quad\Omega\times (0, T],&\\
			\label{eq_mu_con}
			\mu + \epsilon\Delta \phi- \epsilon^{-1} \left(\phi^3 - \phi\right)=0,\quad \text{in} \quad\Omega\times (0, T],&\\
			\label{eq_ns_con}
			\frac{\partial \boldsymbol{u}}{\partial t}+(\boldsymbol{u}\cdot\nabla)\boldsymbol{u}-\nu\Delta\boldsymbol{u}+\nabla p-\gamma \mu\nabla\phi=0,\quad\text{in} \quad\Omega\times (0, T],&\\
			\label{eq_incompressi}
			\nabla\cdot\boldsymbol{u}=0,\quad\text{in} \quad\Omega\times (0, T],&
		\end{align}
	\end{subequations}
	with the following boundary and initial conditions 
	\begin{subequations}
		\label{eq_boundary_initial_conditions_equations}
		\begin{align}
			\label{eq_boundary_conditions_equation}
			\frac{\partial \phi}{\partial \mathbf{n}}=\frac{\partial\mu}{\partial \mathbf{n}}=0,\quad\boldsymbol{u}=0,&\quad\text{on}\quad \partial\Omega\times (0,T],\\
			\label{eq_initial_conditions_equation}
			\phi(\mathbf{x},0)=\phi^0,\quad\boldsymbol{u}(\mathbf{x},0)=\boldsymbol{u}^0,&\quad\text{in}\quad \Omega.
		\end{align}
	\end{subequations}
	Here $\Omega\subset\mathbb{R}^d,~(d=2,3)$ is a bounded convex polygonal or polyhedral domain, and for $t\in (0,T]$.
	In this model, $\phi$ represents the phase field variable, $\mu$ denotes the chemical potential, $\boldsymbol{u}$ and $p$ are the velocity and pressure fields of the fluid, respectively.
	$\epsilon$ is the interfacial width between the two phases field, and $M(\phi) > 0$ is the mobility, $\nu = 1/Re$ where $Re$ is the Reynolds number, $\gamma = 1/We^*$ and $We^*$ is the modified Weber number that measures the relative strength of kinetic and surface energies.
    We denote $\phi^3 - \phi =: f(\phi) = F^{\prime}(\phi)$, then the following condition  holds\cite{2010_ShenJie_Energy_stable_schemes_for_Cahn_Hilliard_phase_field_model_of_two_phase_incompressible_flows}: there exists a positive constant $L$ such that 
	\begin{equation}\label{eq_phi_condition}
		\max_{\phi\in \mathbb{R}}~|f'(\phi)|\leq L.
	\end{equation}
	\indent Over the past decade, numerous numerical schemes have been proposed by researchers for the CHNS model, each addressing different physical properties. (\rmnum{1}) For energy-stable discretization schemes, techniques primarily developed for the Cahn–Hilliard equation have been adopted, such as convex splitting\cite{2017_Diegel_Convergence_analysis_and_error_estimates_for_a_second_order_accurate_finite_element_method_for_the_Cahn_Hilliard_Navier_Stokes_system}, linear stabilization techniques\cite{2017_CaiYongyong_Error_estimates_for_time_discretizations_of_Cahn_Hilliard_and_Allen_Cahn_phase_field_models_for_two_phase_incompressible_flows}, the invariant energy quadratization approach\cite{2017_YangXiaofeng_Numerical_approximations_for_the_molecular_beam_epitaxial_growth_model_based_on_the_invariant_energy_quadratization_method}, and scalar auxiliary variable methods\cite{2018_Shenjie_Xujie_CAEAFTSAVSYGF,2018_Shenjie_Xujie_SAV,2025_Gao_CAMWA,2025_Gao_NMTMA}; (\rmnum{2}) computational efficiency-enhanced, including decoupling algorithms\cite{2015_ShenJie_Decoupled_energy_stable_schemes_for_phase_field_models_of_two_phase_incompressible_flows,2021_LiMinghui_New_efficient_time_stepping_schemes_for_the_Navier_Stokes_Cahn_Hilliard_equations} and reduced basis methods\cite{2025_Abbaszadeh_A_reduced_order_least_squares_support_vector_regression_and_isogeometric_collocation_method_to_simulate_Cahn_Hilliard_Navier_Stokes_equation}; (\rmnum{3}) high-order time discretizations\cite{2022JieShen_LiXiaoliMSAVCHNStwo_phase_incompressible_flows,2025_LiuXin_CHNSSDC}. Nevertheless, due to the inherent energy dissipation law, strong nonlinearity, and multiscale characteristics of the CHNS system, numerical schemes that combine low computational complexity, high-order temporal accuracy, and unconditional energy stability remain worthy of further investigation.  \\
	\indent In recent years, the novel time filtering (TF) technique has garnered significant attention due to its obvious advantages: not only can it improve a low-order time-discrete scheme to second-order or even higher precision with minimal computational overhead, but it also serves as an efficient, low-cost error estimator, providing critical support for adaptive time-step strategies.
	So, taking the variable time-step backward Euler (BE) method as an example, and denote stepsize as $\Delta t_n := t_{n+1} - t_n$ for discrete time $t_{n+1},t_n$, and stepsize ratio $\omega_n:= \Delta t_{n+1}/\Delta t_n$, then the main technique of TF can be expressed
	\begin{itemize}
		\item \textbf{Step 1} : Backward Euler
		$$
		\frac{y^{n+1} - y^n}{\Delta t_n} = f(t^{n+1}, y^{n+1}).
		$$
		\item \textbf{Step 2} : Time Filtering
		$$
		y^{n+1}_{TF} = y^{n+1} - \frac{\omega_n}{1+2\omega_n}\left(y^{n+1} - (1+\omega_n)y^n + \omega_ny^{n-1}\right).
		$$
		\item \textbf{Step 3} : To 
		\begin{itemize}
			\item improve temporal accuracy: $y^{n+1} = y^{n+1}_{TF}$, or
			\item serve as error estimator to adjust $\Delta t_{n+1}$ or $\Delta t_{n}$: $estimator = \|y^{n+1}_{TF} - y^{n+1}\|$.
		\end{itemize}
	\end{itemize}
	The TF technique was initially proposed by Kwizak and Robert\cite{1971_ASEMIIMPLICITSCHEMEFORGRIDPOINTATMOSPHERICMODELSOFTHEPRIMITIVEEQUATIONS}, subsequently subjected to theoretical analysis by Asselin et al. \cite{1972_FrequencyFilterforTimeIntegrations}, and further improved by Williams \cite{2009_AProposedModificationtotheRobertAsselinTimeFilter,2010_TheRAWFilterAnImprovementtotheRobertAsselinFilterinSemiImplicitIntegrations}. Since then, it has been widely applied to various models, such as the Navier-Stokes equations \cite{2020_DeCaria_A_time_accurate_adaptive_discretization_for_fluid_flow_problems,2023_FengXinlong_Filtered_time_stepping_method_for_incompressible_Navier_Stokes_equations_with_variable_density}, Stokes-Darcy model \cite{2023_QinYi_Analysis_of_a_new_adaptive_time_filter_algorithm_for_the_unsteady_Stokes_Darcy_model}, natural convection problems\cite{2023_FengXinlong_Analysis_of_a_filtered_time_stepping_finite_element_method_for_natural_convection_problems,2025_FengXinlong_An_analysis_of_second_order_sav_filtered_time_stepping_finite_element_method_for_unsteady_natural_convection_problems}, etc. However, due to the multi-variable nature, strong coupling effects, and high nonlinearity inherent in the CHNS model, the theoretical analysis of numerical schemes incorporating the TF technique remains particularly challenging and complex.\\
	 \indent In this work, we introduce the TF technique into the CHNS  model to construct several novel semi-discrete schemes that are first-order or second-order in time. Firstly, we employ the first-order BE method to discretize the CHNS model, where $f(\phi)$ is treated either fully implicitly or explicitly with a linear stabilization term. Subsequently, we utilize the TF technique to elevate the temporal accuracy from first-order to second-order. Furthermore, we theoretically analyze and prove the unconditional energy stability and second-order temporal error estimates of the proposed schemes. \\
	 \indent The main research conclusions of this work are outlined as follows: Our newly formulated second-order scheme accommodates constant, variable, and adaptive time-step alike. It offers minimal implementation complexity, necessitating only a single line of code modification to upgrade from the existing first-order scheme. Consequently, second-order temporal accuracy is achieved for all variables at a negligible computational expense. We rigorously demonstrate the unconditional energy stability of this numerical discretization approach. In comparison to energy-based adaptive schemes and conventional higher-order time semi-discretization methods, the error estimators employed in our adaptive time-stepping semi-discrete schemes significantly reduce computational overhead.\\
	\indent The rest of work is organized as follows.
	Some notations and preliminaries are given in Section \ref{section-notations-preliminaries}.
	The time-filtered scheme is presented in  Section \ref{section-3} for the CHNS model and then, the unconditional energy stability is proved. In Section 
	\ref{section_numerical_analysis}, the error estimations are analyzed, and finally, 
	in Section \ref{section_numerical_tests}, theoretical analyses of the time-filtered scheme are confirmed by several numerical experiments.
	\section{Notations and preliminaries}\label{section-notations-preliminaries}
	We introduce some notations and inequalities in this section. Here we use the notation $H^m(\Omega)$ and $(\cdot,\cdot)_m$, $\|\cdot\|_m$, for some positive integer $m$, to denote the standard Sobolev space $W^{m,2}(\Omega)$ and its inner-product and norm. $H^{-1}(\Omega) := (H^1(\Omega))^*$ denote the dual space of $H^1(\Omega)$, and $\langle\cdot, \cdot\rangle$ is the duality paring between $H^{-1}(\Omega)$ and $H^{1}(\Omega)$. The $\|\cdot\|$ and $\left(\cdot,\cdot\right)$ denote the norm and inner product of $L^2$ space, and denote
	$$
	\mathring{H}^1(\Omega) := H^1(\Omega) \cap L_0^2(\Omega), \quad \mathring{H}^{-1}(\Omega) := \{v \in H^{-1}(\Omega) | \langle v,1 \rangle = 0\},
	$$
	with $L_0^2(\Omega)$ representing those functions in $L^2(\Omega)$ with zero mean. In particular, the phase function and chemical potential space $X$ and velocity space $\boldsymbol{Y}$ and pressure space $Q$ are
	defined 
	$$
	\begin{aligned}
		&X:=H^1(\Omega),~\boldsymbol{Y}:=H_0^1(\Omega)^d, \text{and } Q:=L_0^2(\Omega).
	\end{aligned}
	$$
	We choose the velocity and pressure spaces $\boldsymbol{Y}, Q$ that satisfy the inf-sup condition:
	\begin{equation*}\label{eq_inf_sup_condition}
		\inf _{q\in Q}\sup_{{\boldsymbol{v}}\in \boldsymbol{Y}} \frac{\left(q,\nabla{\boldsymbol{v}}\right)}{\|q\|\|\nabla{\boldsymbol{v}}\|}\geq \beta\geq 0,
	\end{equation*}
	for some constant $\beta$. \\	
	\indent Following \cite{2015_Diegel_Analysis_of_a_mixed_finite_element_method_for_a_Cahn_Hilliard_Darcy_Stokes_system}, we define a linear operator for later analysis. Let a linear operator $\mathcal{T}: \mathring{H}^{-1}(\Omega) \rightarrow \mathring{H}^{1}(\Omega)$ via the following variational problem: given $\zeta \in \mathring{H}^{-1}(\Omega)$ such that
	$$
	(\nabla \mathcal{T}(\zeta), \nabla \chi) = \langle \zeta, \chi \rangle, \quad \forall \chi \in \mathring{H}^{1}(\Omega).
	$$
	The linear operator $\mathcal{T}$ is well-defined via the Riesz representation theorem, and the following lemma has been established.
	\begin{Lemma}[\cite{2015_Diegel_Analysis_of_a_mixed_finite_element_method_for_a_Cahn_Hilliard_Darcy_Stokes_system}]
		Let \( \zeta, \xi \in \mathring{H}^{-1}(\Omega) \) and, for such functions, set
		\[
		(\zeta, \xi)_{-1} := (\nabla \mathcal{T}(\zeta), \nabla \mathcal{T}(\xi)) = \langle \zeta, \mathcal{T}(\xi) \rangle = \langle \mathcal{T}(\zeta), \xi \rangle,
		\]
		and the induced norm is equal to the operator norm:
		\[
		\|\zeta\|_{-1} := \sqrt{(\zeta, \zeta)_{-1}} = \sup_{0 \neq \chi \in \mathring{H}^{1}} \frac{\langle \zeta, \chi \rangle}{\|\nabla \chi\|}.
		\]
		Consequently, for all \(\chi \in H^{1}(\Omega)\) and all \(\zeta \in \mathring{H}^{-1}(\Omega)\),
		\[
		|\langle \zeta, \chi \rangle| \leq \|\zeta\|_{-1} \|\nabla \chi\|.
		\]
		Furthermore, for all \(\zeta \in L^{2}_{0}(\Omega)\), we have the Poincar\'{e} type inequality
		\[
		\|\zeta\|_{-1} \leq C \|\zeta\|,
		\]
		where \( C > 0 \) is the usual Poincar\'{e} constant.
	\end{Lemma}
    For the trilinear term $((\boldsymbol{u} \cdot \nabla )\boldsymbol{v},\boldsymbol{w})$, we have the following estimates.    
	\begin{Lemma}[\cite{2021_lixiaoli_New_SAV_NS}]
		\label{lemma_nonlinear_term_boundness}
		There exists $C>0$ such that
		\begin{flalign}
			\label{eq_ns_nonlinear_inequalities01}
			((\boldsymbol{u} \cdot \nabla )\boldsymbol{v},\boldsymbol{w})\leq\left\{\begin{array}{l}
				C\|\boldsymbol{u}\|_{1}\|\boldsymbol{v}\|_{1}\|\boldsymbol{w}\|_{1},\\
				C\|\boldsymbol{u}\|_{2}\|\boldsymbol{v}\|\|\boldsymbol{w}\|_{1},\\
				C\|\boldsymbol{u}\|_{2}\|\boldsymbol{v}\|_{1}\|\boldsymbol{w}\|,\\
				C\|\boldsymbol{u}\|_{1}\|\boldsymbol{v}\|_{2}\|\boldsymbol{w}\|,\\
				C\|\boldsymbol{u}\|\|\boldsymbol{v}\|_{2}\|\boldsymbol{w}\|_{1}.
			\end{array}\right.
		\end{flalign}
	\end{Lemma}
    Moreover, if we denote $\boldsymbol{W} := \{\boldsymbol{v} \in \boldsymbol{L}^2(\Omega) | \nabla \cdot \boldsymbol{v} = 0, \; \boldsymbol{v} \cdot \mathbf{n} = 0\}$, then we have the skew-symmetric property:
    \begin{equation}
    	((\boldsymbol{u} \cdot \nabla )\boldsymbol{v},\boldsymbol{v}) = 0, \quad \boldsymbol{u} \in \boldsymbol{W}, \; \boldsymbol{v} \in \boldsymbol{H}^1_0(\Omega).
    \end{equation}
    \indent We will use the following discrete version of the Gr\"{o}nwall inequality in the later numerical analysis.
	\begin{Lemma}[\cite{1990_Heywood_Finite_element_approximation_of_the_nonstationary_Navier_Stokes_problem_IV_Error_analysis_for_second_order_time_discretization}]
		\label{lemma_gronwall_inequalities}
		For all $n\geq 0$ and $N\geq 0$, let $\Delta t$, $C$, $a_n$, $b_n$, $c_n$, $d_n$ be non-negative numbers such that
		\begin{equation*}
			\label{eq_gronwall_inequality_0001}
			a_N+\Delta t\sum_{n=0}^{N}b_n\leq \Delta t\sum_{n=0}^{N}d_na_n+\Delta t\sum_{n=0}^{N}c_n+C,
		\end{equation*} 
	and suppose $\Delta td_n \leq 1$, s
	\begin{equation*}
		\label{eq_gronwall_inequality_0002}
		a_N+\Delta t\sum_{n=0}^{N}\leq \exp\left(\Delta t\sum_{n=0}^{N}\frac{1}{1-\Delta td_n}\right)\left(\Delta t\sum_{n=0}^{N}c_n+C\right).
	\end{equation*}
	\end{Lemma}
	Throughout the manuscript, we use $C$ or $c$, with or without a subscript, to denote a positive constant independent of discretization parameters, which could have uncertain values in different places.
	
	\section{Time-filtered scheme for constant and variable time-step}\label{section-3}
	Dividing the time domain $(0,T]$ into $n$ intervals $[t_n, t_{n+1}]$, $n=0,1,2, \cdots, N-1$ where $t_0 = 0$, $t_N = T$, and defining $\Delta t_n := t_{n+1} - t_n$, and time-step ratio $\omega_n:= \Delta t_{n+1}/\Delta t_n$. Specifically, for the constant time-step, we denote $\Delta t_n = \Delta t$ for $\forall n$.  \\
	\indent We first present the classical first-order nonlinear and linear implicit Euler schemes with constant time-step, where $f(\phi)$ is handled full-implicitly, or explicitly with stabilized term, respectively. These two schemes will be used as parts of the second-order time-filtered schemes. That is, 

	\noindent \textbf{Algorithm 1} \emph{(Fully Implicit Backward Euler (FIMBE))}. Given $\phi^{n-1},$ $\phi^{n}$, $\mu^{n-1},$ $\mu^{n}$, $\boldsymbol{u}^{n-1}$, $\boldsymbol{u}^{n}$, $p^n$, find $(\phi^{n+1}$, $\mu^{n+1}$,  $\boldsymbol{u}^{n+1},p^{n+1}) $, such that 
	\begin{equation}
		\label{eq_BE_scheme0001}
		\left\{\begin{aligned}
			\frac{\phi^{n+1} - \phi^{n}}{\Delta t} + \boldsymbol{u}^{n+1} \cdot \nabla \phi^{n+1} - \epsilon M \Delta \mu^{n+1} &= 0 , \\
			- \mu^{n+1} - \epsilon \Delta\phi^{n+1} + \frac{1}{\epsilon}\left((\phi^{n+1})^3-\phi^{n+1}\right) &= 0 , \\
			\frac{\boldsymbol{u}^{n+1} - \boldsymbol{u}^{n}}{\Delta t} - \nu \Delta \boldsymbol{u}^{n+1} + (\boldsymbol{u}^{n+1} \cdot \nabla) \boldsymbol{u}^{n+1} + \nabla p^{n+1} - \gamma \mu^{n+1}\nabla\phi^{n+1} &= 0 , \\
			\nabla \cdot \boldsymbol{u}^{n+1} &= 0 .
		\end{aligned}\right.
	\end{equation}
    
    \indent And the classical linear semi-implicit scheme with stabilization is presented as follows.
    
	\noindent \textbf{Algorithm 2} \emph{(Semi-Implicit Backward Euler (SIMBE))}. Given $\phi^{n-1},$ $\phi^{n}$, $\mu^{n-1},$ $\mu^{n}$, $\boldsymbol{u}^{n-1}$, $\boldsymbol{u}^{n}$, $p^n$, find $(\phi^{n+1}$, $\mu^{n+1}$,  $\boldsymbol{u}^{n+1},p^{n+1}) $, such that 
	\begin{equation}
		\label{eq_BE_scheme0002}
		\left\{\begin{aligned}
			\frac{\phi^{n+1} - \phi^{n}}{\Delta t} + \boldsymbol{u}^{n+1} \cdot \nabla \phi^{n} - \epsilon M \Delta \mu^{n+1} &= 0 , \\
			- \mu^{n+1} - \epsilon \Delta\phi^{n+1} + \frac{S}{\epsilon}\left(\phi^{n+1} - \phi^n\right) + \frac{1}{\epsilon}\left((\phi^{n})^3-\phi^{n}\right) &= 0 , \\
			\frac{\boldsymbol{u}^{n+1} - \boldsymbol{u}^{n}}{\Delta t} - \nu \Delta \boldsymbol{u}^{n+1} + (\boldsymbol{u}^{n} \cdot \nabla) \boldsymbol{u}^{n+1} + \nabla p^{n+1} - \gamma \mu^{n+1}\nabla\phi^{\star} &= 0 , \\
			\nabla \cdot \boldsymbol{u}^{n+1} &= 0 .
		\end{aligned}\right.
	\end{equation}
    where $S$ is the stabilization parameter. Then, to clearly demonstrate the simplicity of the TF technique, we incorporate the TF into the fully implicit scheme, thereby deriving the following nonlinear fully implicit Euler TF scheme:

	\noindent \textbf{Algorithm 3} \emph{(Fully Implicit Backward Euler with Time Filtering (FIMBE-TF))}. Given $\phi^{n-1},$ $\phi^{n}$, $\mu^{n-1},$ $\mu^{n}$, $\boldsymbol{u}^{n-1}$, $\boldsymbol{u}^{n}$, $p^n$, find $\left(\phi^{n+1},\mu^{n+1},\boldsymbol{u}^{n+1},p^{n+1}\right) $:
	
	\noindent \textbf{Step 3.1.} Solve
	\begin{equation}
		\label{eq_FIBETF_scheme0001}
		\left\{\begin{aligned}
			\frac{\tilde{\phi}^{n+1} - \phi^{n}}{\Delta t} + \tilde{\boldsymbol{u}}^{n+1} \cdot \nabla \tilde{\phi}^{n+1} - \epsilon M \Delta \tilde{\mu}^{n+1} &= 0 , \\
			- \tilde{\mu}^{n+1} - \epsilon \Delta\tilde{\phi}^{n+1} + \frac{1}{\epsilon}\left((\tilde{\phi}^{n+1})^3 - \tilde{\phi}^{n+1}\right) &= 0 , \\
			\frac{\tilde{\boldsymbol{u}}^{n+1} - \boldsymbol{u}^{n}}{\Delta t} - \nu \Delta \tilde{\boldsymbol{u}}^{n+1} + (\tilde{\boldsymbol{u}}^{n+1} \cdot \nabla) \tilde{\boldsymbol{u}}^{n+1} + \nabla \tilde{p}^{n+1} - \gamma \tilde{\mu}^{n+1}\nabla\tilde{\phi}^{n+1} &= 0 , \\
			\nabla \cdot \tilde{\boldsymbol{u}}^{n+1} &= 0 .
		\end{aligned}\right.
	\end{equation}
	\textbf{Step 3.2.} Apply time filtering
	\begin{equation}
		\label{eq_FIBETF_scheme0002}
		\begin{aligned}
			\phi^{n+1} &= \tilde{\phi}^{n+1} - \frac13\left(\tilde{\phi}^{n+1} - 2\phi^{n} + \phi^{n-1}\right),  \\
			\mu^{n+1}  &= \tilde{\mu}^{n+1} - \frac13\left(\tilde{\mu}^{n+1} - 2\mu^{n} + \mu^{n-1}\right),  \\
			\boldsymbol{u}^{n+1} &= \tilde{\boldsymbol{u}}^{n+1} - \frac13\left(\tilde{\boldsymbol{u}}^{n+1} - 2\boldsymbol{u}^{n} + \boldsymbol{u}^{n-1}\right),  \\
			p^{n+1}    &= \tilde{p}^{n+1} - \frac13\left(\tilde{p}^{n+1} - 2 p^{n} + p^{n-1}\right),   \text{ for Option A},\\
			&\text{or}  \\
			p^{n+1}    &= \tilde{p}^{n+1}, \text{ for Option B}.
		\end{aligned}
	\end{equation}
    \begin{Remark}
    	In practical implementations, one can adapt Crank-Nicholson method to initialize the above and newly constructed schemes below.
    \end{Remark}

    For the purpose of efficient implementation, we employ the extrapolation technique to linearize \textbf{Algorithm 3} to derive the following linear semi-implicit Euler TF scheme, i.e., 
	
	\noindent \textbf{Algorithm 4}\emph{ (Semi-Implicit Backward Euler with Time Filtering (SIMBE-TF))}. Given $\phi^{n-1},$ $\phi^{n}$, $\mu^{n-1},$ $\mu^{n}$, $\boldsymbol{u}^{n-1}$, $\boldsymbol{u}^{n}$, $p^n$, find $\left(\phi^{n+1},\mu^{n+1},\boldsymbol{u}^{n+1},p^{n+1}\right) $:
	
	\noindent \textbf{Step 4.1.} Solve
	\begin{equation}\label{eq_algorithm3_1}
		\left\{\begin{aligned}
			\frac{\tilde{\phi}^{n+1} - \phi^{n}}{\Delta t} + \tilde{\boldsymbol{u}}^{n+1} \cdot \nabla \bar{\phi}^{n+1} - \epsilon M \Delta \tilde{\mu}^{n+1} &= 0 , \\
			- \tilde{\mu}^{n+1} - \lambda \Delta\tilde{\phi}^{n+1} + \frac{S\Delta t}{\epsilon}\left(\tilde{\phi}^{n+1} - \phi^{n}\right) + \frac{1}{\epsilon}\left(2f(\phi^n)-f(\phi^{n-1})\right) &= 0 , \\
			\frac{\tilde{\boldsymbol{u}}^{n+1} - \boldsymbol{u}^{n}}{\Delta t} - \nu \Delta \tilde{\boldsymbol{u}}^{n+1} + (\bar{\boldsymbol{u}}^{n+1} \cdot \nabla) \tilde{\boldsymbol{u}}^{n+1} + \nabla \tilde{p}^{n+1} - \gamma \tilde{\mu}^{n+1}\nabla\bar{\phi}^{n+1} &= 0 , \\
			\nabla \cdot \tilde{\boldsymbol{u}}^{n+1} &= 0 .
		\end{aligned}\right.
	\end{equation}
	\textbf{Step 4.2.} Apply time filtering
	\begin{equation}\label{eq_algorithm3_2}
		\begin{aligned}
			\phi^{n+1} &= \tilde{\phi}^{n+1} - \frac13\left(\tilde{\phi}^{n+1} - 2\phi^{n} + \phi^{n-1}\right),  \\
			\mu^{n+1}  &= \tilde{\mu}^{n+1} - \frac13\left(\tilde{\mu}^{n+1} - 2\mu^{n} + \mu^{n-1}\right),  \\
			\boldsymbol{u}^{n+1} &= \tilde{\boldsymbol{u}}^{n+1} - \frac13\left(\tilde{\boldsymbol{u}}^{n+1} - 2\boldsymbol{u}^{n} + \boldsymbol{u}^{n-1}\right),  \\
			p^{n+1}    &= \tilde{p}^{n+1} - \frac13\left(\tilde{p}^{n+1} - 2 p^{n} + p^{n-1}\right),  \text{ for Option A},\\
			&\text{or}  \\
			p^{n+1}    &= \tilde{p}^{n+1},  \text{ for Option B}.
		\end{aligned}
	\end{equation}
	where $\bar{g}^{n+1} := 2 g^n - g^{n-1}$ for any sequence $\{g^n\}$. \\
	
	Next, we construct the variable time-step backward Euler with time filtering.
	 
	\noindent \textbf{Algorithm 5}\emph{ (Variable Stepsize Backward Euler with Time Filtering (VSBE-TF))}. Given $\phi^{n-1},$ $\phi^{n}$, $\mu^{n-1},$ $\mu^{n}$, $\boldsymbol{u}^{n-1}$, $\boldsymbol{u}^{n}$, $p^n$, find $\left(\phi^{n+1},\mu^{n+1},\boldsymbol{u}^{n+1},p^{n+1}\right) $:
	
	\noindent \textbf{Step 5.1.} Solve
	\begin{equation}\label{algorithm_variable_stepsize_1}
		\left\{\begin{aligned}
			\frac{\tilde{\phi}^{n+1} - \phi^{n}}{\Delta t_n} + \tilde{\boldsymbol{u}}^{n+1} \cdot \nabla \tilde{\phi}^{n+1} - \epsilon M \Delta \tilde{\mu}^{n+1} &= 0 , \\
			- \tilde{\mu}^{n+1} - \epsilon \Delta\tilde{\phi}^{n+1} + \frac{1}{\epsilon}\left((\tilde{\phi}^{n+1})^3 - \tilde{\phi}^{n+1}\right) &= 0 , \\
			\frac{\tilde{\boldsymbol{u}}^{n+1} - \boldsymbol{u}^{n}}{\Delta t_n} - \nu \Delta \tilde{\boldsymbol{u}}^{n+1} + (\tilde{\boldsymbol{u}}^{n+1} \cdot \nabla) \tilde{\boldsymbol{u}}^{n+1} + \nabla \tilde{p}^{n+1} - \gamma \tilde{\mu}^{n+1}\nabla\tilde{\phi}^{n+1} &= 0 , \\
			\nabla \cdot \tilde{\boldsymbol{u}}^{n+1} &= 0 .
		\end{aligned}\right.
	\end{equation}
    \textbf{Step 5.2.} Apply time filtering
    \begin{equation}\label{algorithm_variable_stepsize_2}
    	\begin{aligned}
    		\phi^{n+1} &= \tilde{\phi}^{n+1} + \frac{\omega_n}{1+2\omega_n}\left(\tilde{\phi}^{n+1} - (1+\omega_n)\phi^{n} + \omega_n\phi^{n-1}\right),  \\
    		\mu^{n+1}  &= \tilde{\mu}^{n+1} + \frac{\omega_n}{1+2\omega_n}\left(\tilde{\mu}^{n+1} - (1+\omega_n)\mu^{n} + \omega_n\mu^{n-1}\right),  \\
    		\boldsymbol{u}^{n+1} &= \tilde{\boldsymbol{u}}^{n+1} + \frac{\omega_n}{1+2\omega_n}\left(\tilde{\boldsymbol{u}}^{n+1} - (1+\omega_n)\boldsymbol{u}^{n} + \omega_n\boldsymbol{u}^{n-1}\right),  \\
    		p^{n+1}    &= \tilde{p}^{n+1} + \frac{\omega_n}{1+2\omega_n}\left(\tilde{p}^{n+1} + (1+\omega_n)p^{n} + \omega_np^{n-1}\right).
    	\end{aligned}
    \end{equation}
    \begin{Remark}
    	The above nonlinear scheme can be linearized by second-order extrapolation with variable stepsize; i.e., replacing with $g^{n+1} \approx (1+\Delta t_n/\Delta t_{n-1})g^n - \Delta t_n/\Delta t_{n-1}g^{n-1} = (1+\omega_n)g^n - \omega_ng^{n-1}$ for the appropriate variables in the nonlinear coupled terms.
    \end{Remark}

    Furthermore, we propose the following adaptive time-step backward Euler with time filtering.
    
    \noindent \textbf{Algorithm 6}\emph{ (Adaptive Stepsize Backward Euler with Time Filtering (ASBE-TF))}. Given $\phi^{n-1},$ $\phi^{n}$, $\mu^{n-1},$ $\mu^{n}$, $\boldsymbol{u}^{n-1}$, $\boldsymbol{u}^{n}$, $p^n$, and $g^{\star} := (1+\omega_n)g^n - \omega_ng^{n-1}$ with $g = \boldsymbol{u},\phi$, and $\gamma_1,\gamma_2$ are two parameters. Then, find $\left(\phi^{n+1},\mu^{n+1},\boldsymbol{u}^{n+1},p^{n+1}\right)$, such that:
    
    \noindent \textbf{Step 6.1.} Solve
    \begin{equation}\label{algorithm_adaptive_stepsize_1}
    	\left\{\begin{aligned}
    		\frac{\phi^{n+1} - \phi^{n}}{\Delta t_n} + \boldsymbol{u}^{n+1} \cdot \nabla \phi^{\star} - \epsilon M \Delta \mu^{n+1} &= 0 , \\
    		- \mu^{n+1} - \epsilon \Delta\phi^{n+1} + \frac{1}{\epsilon} f \left(\phi^{\star}\right) &= 0 , \\
    		\frac{\boldsymbol{u}^{n+1} - \boldsymbol{u}^{n}}{\Delta t_n} - \nu \Delta \boldsymbol{u}^{n+1} + (\boldsymbol{u}^{\star} \cdot \nabla) \boldsymbol{u}^{n+1} + \nabla p^{n+1} - \gamma \mu^{n+1}\nabla\phi^{\star} &= 0 , \\
    		\nabla \cdot \boldsymbol{u}^{n+1} &= 0 .
    	\end{aligned}\right.
    \end{equation}
    \textbf{Step 6.2.} Apply time filtering to find estimate indicator:
    \begin{equation}
    	\begin{aligned}
    		EST_{\phi} &:= \left\|\phi^{n+1}_{TF} - \phi^{n+1}\right\|,  \\
    		EST_{\boldsymbol{u}} &:= \left\|\boldsymbol{u}^{n+1}_{TF} - \boldsymbol{u}^{n+1}\right\|,  \\
    	\end{aligned}
    \end{equation}
    where
    \begin{equation}\label{algorithm_adaptive_stepsize_2}
    	\begin{aligned}
    		\phi^{n+1}_{TF} &:= \phi^{n+1} + \frac{\omega_n}{1+2\omega_n}\left(\phi^{n+1} - (1+\omega_n)\phi^{n} + \omega_n\phi^{n-1}\right),  \\
    		\boldsymbol{u}^{n+1}_{TF} &:= \boldsymbol{u}^{n+1} + \frac{\omega_n}{1+2\omega_n}\left(\boldsymbol{u}^{n+1} - (1+\omega_n)\boldsymbol{u}^{n} + \omega_n\boldsymbol{u}^{n-1}\right).  \\
    	\end{aligned}
    \end{equation}
    \textbf{Step 6.3.} Determine the stepsize by:
    \begin{itemize}
    	\item if $\max\{EST_{\phi}, EST_{\boldsymbol{u}}\} < tol$, then
    	$$
    	\Delta t_{n+1} = \min\left\{2\Delta t_{n}, \gamma_1\Delta t_{n}\min\left\{\left(\frac{tol}{EST_{\phi}}\right)^{\frac12},\left(\frac{tol}{EST_{\boldsymbol{u}}}\right)^{\frac12}\right\}\right\}.
    	$$
    	\item if $\max\{EST_{\phi}, EST_{\boldsymbol{u}}\} \geq tol$, then
    	$$
    	\Delta t_{n} = \max\left\{0.5\Delta t_{n}, \gamma_2\Delta t_{n}\min\left\{\left(\frac{tol}{EST_{\phi}}\right)^{\frac12},\left(\frac{tol}{EST_{\boldsymbol{u}}}\right)^{\frac12}\right\}\right\},
    	$$
    	and recompute the above steps.
    \end{itemize}
    
    \begin{Remark}
    	Typically, the two parameters can be chosen as $\gamma_1=0.9$, $\gamma_2 = 0.7$. See \cite{2020_Layton_William_A_time_accurate_adaptive_discretization_for_fluid_flow_problems,2021_DeCaria_A_variable_stepsize_variable_order_family_of_low_complexity} for the specific roles and explanations.
    \end{Remark}
    \indent For the several TF time-stepping schemes constructed above, we analyze the linear, constant stepsize scheme \eqref{eq_algorithm3_1}-\eqref{eq_algorithm3_2}, and the analysis for the nonlinear scheme \eqref{eq_FIBETF_scheme0001}-\eqref{eq_FIBETF_scheme0002} can be derived similarly. The numerical analysis of the variable time-step scheme \eqref{algorithm_variable_stepsize_1}-\eqref{algorithm_variable_stepsize_2} will be presented in the future.
    
	\subsection{Energy stability}
	In the subsection, we will prove unconditional energy stabilization for the time semi-discrete scheme \eqref{eq_algorithm3_1}-\eqref{eq_algorithm3_2}. To analyze, we denote
	\begin{equation*}
		\begin{aligned}
			\mathcal{A}(s^{n+1}) &:= \tilde{s}^{n+1} - s^{n} = \frac32s^{n+1} - 2s^{n} + \frac12s^{n-1}, \\ 
			\mathcal{B}(s^{n+1}) &:= \tilde{s}^{n+1} = \frac32s^{n+1} - s^{n} + \frac12s^{n-1},
		\end{aligned}
	\end{equation*}
	with $s=\phi$, or $\mu$, or $\boldsymbol{u}$, or $p$. 
	Therefore, we can get
	\begin{equation}
		\label{eq_A_B_form}
		\begin{aligned}
			\left(\mathcal{A}(s^{n+1}),\mathcal{B}(s^{n+1}) \right)=&\left(\frac32s^{n+1} - 2s^{n} + \frac12s^{n-1},\frac32s^{n+1} - s^{n} + \frac12s^{n-1}\right)\\
			=&\frac{1}{4}\left(\left(3s^{n+1}-4s^n+s^{n-1},2s^{n+1}\right)+\left(3s^{n+1}-4s^n+s^{n-1},s^{n+1}-2s^{n}+s^{n-1}\right)\right)\\
			=&\frac{1}{4}\left(\|s^{n+1}\|^2+\|2s^{n+1}-s^n\|^2+\|s^{n+1}-s^n\|^2\right)\\
			&-\frac{1}{4}\left(\|s^{n}\|^2+\|2s^{n}-s^{n-1}\|^2+\|s^{n}-s^{n-1}\|^2\right)+\frac{3}{4}\left(\|s^{n+1}-2s^n+s^{n-1}\|^2\right).
		\end{aligned}	
	\end{equation}
	Inserting to get the equivalent form of Algorithm 4.
	
	\noindent \textbf{Algorithm 6} \emph{(Equivalent form of Algorithm 4)}. Given $\phi^{n-1},$ $\phi^{n}$, $\mu^{n-1},$ $\mu^{n}$, $\boldsymbol{u}^{n-1}$, $\boldsymbol{u}^{n}$, $p^n$, find $\left(\phi^{n+1},\mu^{n+1},\boldsymbol{u}^{n+1}\right.,$ $\left.p^{n+1}\right) $,
	for $n=1,\cdots,N-1$, satisfying
	\begin{equation}
		\left\{\begin{aligned}
			\frac{\mathcal{A}(\phi^{n+1})}{\Delta t} + \mathcal{B}(\boldsymbol{u}^{n+1}) \cdot \nabla \bar{\phi}^{n+1} - \epsilon M \Delta \mathcal{B}(\mu^{n+1}) &= 0 , \\
			- \mathcal{B}(\mu^{n+1}) - \epsilon \Delta \mathcal{B}(\phi^{n+1}) +\frac{S\Delta t}{\epsilon}\left(\tilde{\phi}^{n+1}-{\phi}^{n}\right)+ \frac{1}{\epsilon} \left(2f(\phi^{n})-f(\phi^{n-1}) \right)&= 0 , \\
			\frac{\mathcal{A}(\boldsymbol{u}^{n+1})}{\Delta t} - \nu \Delta \mathcal{B}(\boldsymbol{u}^{n+1}) + (\bar{\boldsymbol{u}}^{n+1} \cdot \nabla) \mathcal{B}(\boldsymbol{u}^{n+1}) + \nabla \mathcal{B}(p^{n+1}) - \gamma \mathcal{B}(\mu^{n+1})\nabla\bar{\phi}^{n+1} &= 0 , \\
			\nabla \cdot \mathcal{B}(\boldsymbol{u}^{n+1}) &= 0 .
		\end{aligned}\right.
	\end{equation}
    Taking mixed variation, we get
    
	\noindent \textbf{Algorithm 7} \emph{(Equivalent Variational form of Algorithm 4)}. Given $\phi^{n-1},$ $\phi^{n}$, $\mu^{n-1},$ $\mu^{n}$, $\boldsymbol{u}^{n-1}$, $\boldsymbol{u}^{n}$, $p^n$, find $\left(\phi^{n+1},\mu^{n+1},\boldsymbol{u}^{n+1},p^{n+1}\right) $,
	for $n=1,\cdots,N-1$ and $\forall \left(\psi, \omega, \boldsymbol{v}, q\right) \in X \times X \times \boldsymbol{Y} \times Q$, satisfying
	\begin{equation}\label{eq_variational_form_Alg2}
		\left\{\begin{aligned}
			\frac{1}{\Delta t}\left(\mathcal{A}(\phi^{n+1}), \psi\right) + \left(\mathcal{B}(\boldsymbol{u}^{n+1}) \cdot \nabla \bar{\phi}^{n+1}, \psi\right) + \epsilon M\left(\nabla \mathcal{B}(\mu^{n+1}), \nabla\psi\right) &= 0 , \\
			- \left(\mathcal{B}(\mu^{n+1}), \omega\right) + \epsilon\left(\nabla \mathcal{B}(\phi^{n+1}), \nabla \omega\right)+ \frac{S\Delta t}{\epsilon}\left(\tilde{\phi}^{n+1}-{\phi}^{n},\omega\right)+ \frac{1}{\epsilon}\left(2f\left(\phi^{n}\right)-f(\phi^{n-1}), \omega\right) &= 0 , \\
			\frac{1}{\Delta t}\left(\mathcal{A}(\boldsymbol{u}^{n+1}), \boldsymbol{v}\right) + \nu \left(\nabla \mathcal{B}(\boldsymbol{u}^{n+1}), \nabla \boldsymbol{v}\right) + \left((\bar{\boldsymbol{u}}^{n+1} \cdot \nabla) \mathcal{B}(\boldsymbol{u}^{n+1}), \boldsymbol{v}\right) \\
			- \left(\mathcal{B}(p^{n+1}), \nabla\cdot\boldsymbol{v}\right) - \gamma\left(\mathcal{B}(\mu^{n+1})\nabla\bar{\phi}^{n+1}, \boldsymbol{v}\right) &= 0 , \\
			\left(\nabla \cdot \mathcal{B}(\boldsymbol{u}^{n+1}), q\right) &= 0.
		\end{aligned}\right.
	\end{equation}
	\indent For the convenience of the later analysis, we set the parameters $M(\phi),\gamma,\epsilon,\nu$ to 1, though these parameters may be important in the CHNS model. For relevant literature that considers the effect of parameters, e.g., see \cite{2023_Chen_On_the_superconvergence_of_a_hybridizable_discontinuous_Galerkin_method_for_the_Cahn_Hilliard_equation_SINUM,2024_Ma_Error_analysis_with_polynomial_dependence_on_in_SAV_methods_for_the_Cahn_Hilliard_equation_JSC}.  \\	
	\indent We next give a theorem on the unconditional energy stabilization of the time semi-discrete scheme.
	\begin{Theorem}\label{theorem_unconditional_stability}
		For all $\Delta t>0$, 
		assuming that  the condition \eqref{eq_phi_condition} is satisfied with $ S \geq 3 L/\Delta t $. Then, for given  $\phi^0,$ $\phi^1$, $\mu^0,$ $\mu^1$, $\boldsymbol{u}^0$, $\boldsymbol{u}^1$
		and $p^1$, there exists a positive constant $C_0$ such that  $E^0\leq C_0$, and for all $ 0 \leq n \leq N-1$,
		the scheme \eqref{eq_algorithm3_1}-\eqref{eq_algorithm3_2} satisfies
		\begin{equation}
			E^{n+1}-E^{n}\leq 0,
		\end{equation}
	where $E^{n+1}$ is defined by 
	\begin{equation}
		\label{eq_discrete_energy_definition}
		\begin{aligned}
			E^{n+1}=&~\frac{1}{4}\left(\|\nabla\phi^{n+1}\|^2+\|\nabla(2\phi^{n+1}-\phi^n)\|^2+\|\nabla(\phi^{n+1}-\phi^n)\|^2\right)\\
				&~+\frac{1}{4}\left(\|\boldsymbol{u}^{n+1}\|^2+\|2\boldsymbol{u}^{n+1}-\boldsymbol{u}^n\|^2+\|\boldsymbol{u}^{n+1}-\boldsymbol{u}^n\|^2\right) + \frac{3L}{\epsilon} \|\phi^{n+1}-\phi^n\|^2  \\
				&~+ \left(F(\phi^{n+1}),1\right) + \frac{1}{2}\left(F(\phi^{n+1})-F(\phi^n),1\right).
		\end{aligned}
	\end{equation}
    \begin{Remark}
    	We note that one condition appears above regarding the stabilization parameter $S$. On the one hand, this constraint arises from the incorporation of the stabilization term, rather than being induced by the TF technique; Moreover, this type of requirement is a common scenario in the fully explicit discretization with stabilization for the CHNS model. On the other hand, if one wishes to remove this restriction, one can adopt the semi-implicit discretization or fully explicit with scalar auxiliary variables (SAVs) technique.
    \end{Remark}
	\end{Theorem}
	\begin{proof}
		Setting $\left(\psi,\omega,\boldsymbol{v},q\right)=\left(\mathcal{B}(\mu^{n+1}), \mathcal{A}(\phi^{n+1}),\mathcal{B}(\boldsymbol{u}^{n+1}), \mathcal{B}(p^{n+1})\right)$
		in \eqref{eq_variational_form_Alg2},  we have
		\begin{flalign}
			\label{eq_variational_form_inner_product_phi}
			\frac{1}{\Delta t}\left(\mathcal{A}(\phi^{n+1}),\mathcal{B}(\mu^{n+1})\right)+ \left(\mathcal{B}(\boldsymbol{u}^{n+1}) \cdot \nabla \bar{\phi}^{n+1}, \mathcal{B}(\mu^{n+1})\right) + \left(\nabla \mathcal{B}(\mu^{n+1}), \nabla\mathcal{B}(\mu^{n+1})\right) &= 0,\\
			\label{eq_variational_form_inner_product_mu}
			- \left(\mathcal{B}(\mu^{n+1}), \mathcal{A}(\phi^{n+1})\right) + \left(\nabla \mathcal{B}(\phi^{n+1}), \nabla \mathcal{A}(\phi^{n+1})\right) + S\Delta t\left(\tilde{\phi}^{n+1}-{\phi}^{n},\mathcal{A}(\phi^{n+1})\right)&\notag\\
			+ \left(2f\left(\phi^{n}\right)-f(\phi^{n-1}), \mathcal{A}(\phi^{n+1})\right) &= 0 ,
	\end{flalign}
	and
	\begin{flalign}	
		\label{eq_variational_form_inner_product_ns}
		\frac{1}{\Delta t}\left(\mathcal{A}(\boldsymbol{u}^{n+1}), \mathcal{B}(\boldsymbol{u}^{n+1})\right) + \left(\nabla \mathcal{B}(\boldsymbol{u}^{n+1}), \nabla \mathcal{B}(\boldsymbol{u}^{n+1})\right) + \left((\bar{\boldsymbol{u}}^{n+1} \cdot \nabla) \mathcal{B}(\boldsymbol{u}^{n+1}), \mathcal{B}(\boldsymbol{u}^{n+1})\right)\notag \\
		- \left(\mathcal{B}(p^{n+1}), \nabla\cdot\mathcal{B}(\boldsymbol{u}^{n+1})\right) - \left(\mathcal{B}(\mu^{n+1})\nabla\bar{\phi}^{n+1}, \mathcal{B}(\boldsymbol{u}^{n+1})\right) &= 0 , \\
		\label{eq_variational_form_inner_product_incompressible_condition}
		\left(\nabla \cdot \mathcal{B}(\boldsymbol{u}^{n+1}), \mathcal{B}(p^{n+1})\right) &= 0.
	\end{flalign}
	 Multiplying \eqref{eq_variational_form_inner_product_phi} by $\Delta t$ and adding \eqref{eq_variational_form_inner_product_mu},
	 we obtain
	\begin{equation}
		\label{eq_multiplying_adding_two_forms_phi_mu}
		\begin{aligned}
		 \left(\nabla \mathcal{B}(\phi^{n+1}), \nabla \mathcal{A}(\phi^{n+1})\right) + \Delta t
		 	\|\nabla \mathcal{B}(\mu^{n+1})\|^2+\Delta t \left(\mathcal{B}(\boldsymbol{u}^{n+1}) \cdot \nabla \bar{\phi}^{n+1}, \mathcal{B}(\mu^{n+1})\right)&\\
			 + S\Delta t\left(\tilde{\phi}^{n+1}-{\phi}^{n},\mathcal{A}(\phi^{n+1})\right) + \left(2f(\phi^{n})-f(\phi^{n-1}), \mathcal{A}(\phi^{n+1})\right) &=0.
		\end{aligned}
	\end{equation}
    Adding \eqref{eq_variational_form_inner_product_incompressible_condition} and \eqref{eq_variational_form_inner_product_ns}, multiplying  by $\Delta t$, and using skew-symmetry to get
	\begin{equation}
		\label{eq_multiplying_adding_two_forms_ns_con}
		\begin{aligned}
			\left(\mathcal{A}(\boldsymbol{u}^{n+1}), \mathcal{B}(\boldsymbol{u}^{n+1})\right)+
			\Delta t\|\nabla \mathcal{B}(\boldsymbol{u}^{n+1})\|^2-\Delta t\left(\mathcal{B}(\mu^{n+1})\nabla\bar{\phi}^{n+1}, \mathcal{B}(\boldsymbol{u}^{n+1})\right) =0.
		\end{aligned}
	\end{equation}
	According to \eqref{eq_A_B_form}, combining the \eqref{eq_multiplying_adding_two_forms_phi_mu} with \eqref{eq_multiplying_adding_two_forms_ns_con},
	 we have
	\begin{equation}
		\label{eq_combining_the_two_inequalities}
		\begin{aligned}
			&\Delta t	\|\nabla \mathcal{B}(\mu^{n+1})\|^2+
			\frac{1}{4}\left(\|\nabla\phi^{n+1}\|^2+\|\nabla(2\phi^{n+1}-\phi^n)\|^2+\|\nabla(\phi^{n+1}-\phi^n)\|^2\right)\\
			&\quad-\frac{1}{4}\left(\|\nabla\phi^{n}\|^2+\|\nabla(2\phi^{n}-\phi^{n-1})\|^2+\|\nabla(\phi^{n}-\phi^{n-1})\|^2\right)+
			\frac{3}{4}\|\nabla(\phi^{n+1}-2\phi^n+\phi^{n-1})\|^2\\
			&\quad+\frac{1}{4}\left(\|\boldsymbol{u}^{n+1}\|^2+\|2\boldsymbol{u}^{n+1}-\boldsymbol{u}^n\|^2+\|\boldsymbol{u}^{n+1}-\boldsymbol{u}^n\|^2\right) + \Delta t\|\nabla \mathcal{B}(\boldsymbol{u}^{n+1})\|^2\\
			&\quad-\frac{1}{4}\left(\|\boldsymbol{u}^{n}\|^2+\|2\boldsymbol{u}^{n}-\boldsymbol{u}^{n-1}\|^2+\|\boldsymbol{u}^{n}-\boldsymbol{u}^{n-1}\|^2\right)+\frac{3}{4}\left(\|\boldsymbol{u}^{n+1}-2\boldsymbol{u}^n+\boldsymbol{u}^{n-1}\|^2\right)\\
			&\quad+ S\Delta t\left(\tilde{\phi}^{n+1}-\phi^{n},\mathcal{A}(\phi^{n+1})\right) + \left(2f(\phi^{n})-f(\phi^{n-1}), \mathcal{A}(\phi^{n+1})\right) =0.
		\end{aligned}
	\end{equation}
	For the last two term in \eqref{eq_combining_the_two_inequalities}, we firstly have
	\begin{equation}
		\label{eq_Deltat_S}
		\begin{aligned}
			&S\Delta t \left(\tilde{\phi}^{n+1}-\phi^{n},\mathcal{A}(\phi^{n+1})\right)=S\Delta t\|\mathcal{A}(\phi^{n+1})\|^2\\
			=&~\frac{S\Delta t}{4}\|2\left({\phi}^{n+1}-\phi^{n}\right)+\left((\phi^{n+1}-\phi^n)-(\phi^n-\phi^{n-1})\right)\|^2\\
			=&~S\Delta t\|\phi^{n+1}-\phi^n\|^2+\frac{S\Delta t}{4}\|\phi^{n+1}-2\phi^n+\phi^{n-1}\|^2\\
			&~+S\Delta t\left(\phi^{n+1}-\phi^n,(\phi^{n+1}-\phi^n)-(\phi^n-\phi^{n-1})\right)\\
			=&~S\Delta t\|\phi^{n+1}-\phi^n\|^2+\frac{3S \Delta t}{4}\|\phi^{n+1}-2\phi^n+\phi^{n-1}\|^2\\
			&~+\frac{S\Delta t}{2}\left(\|\phi^{n+1}-\phi^n\|^2-\|\phi^n-\phi^{n-1}\|^2\right).
		\end{aligned}
	\end{equation}
	And then, we use the Taylor expansion to obtain
	\begin{equation}\label{eq_Taylor_expansion}
	F({\phi}^{n+1})-F(\phi^n)=f(\phi^n)\left({\phi}^{n+1}-\phi^n\right)+\frac{f'(\xi)}{2}\left({\phi}^{n+1}-\phi^n\right)^2,
	\end{equation}
	for some $\xi\in (\phi^n, \phi^{n+1})$, then we get
	\begin{equation}
		\label{eq_combining_the_two_inequalities_the_last_two_term_001}
		\begin{aligned}
			\left(f(\phi^{n}), \mathcal{A}(\phi^{n+1})\right) = & \left(f(\phi^{n}), \tilde{\phi}^{n+1}-\phi^n\right) =\frac{1}{2} \left(f(\phi^n),3\phi^{n+1}-4\phi^n+\phi^{n-1}\right)\\
			=&\frac{3}{2}\left(f(\phi^n),\phi^{n+1}-\phi^n\right)-\frac{1}{2}\left(f(\phi^n),\phi^n-\phi^{n-1}\right)\\
			=&\frac{3}{2}\left(F(\phi^{n+1})-F(\phi^n),1\right)-\frac{3}{4}\left(f'(\xi_1)\left(\phi^{n+1}-\phi^n\right),\phi^{n+1}-\phi^n\right)\\
			&-\frac{1}{2}\left(F(\phi^n)-F(\phi^{n-1}),1\right)+\frac{1}{4}\left(f'(\xi_2)\left(\phi^n-\phi^{n-1}\right),\phi^n-\phi^{n-1}\right),\\
		\end{aligned}
	\end{equation}
	for some $\xi_1,\xi_2$. On the other hand,
	\begin{equation}
		\label{eq_combining_the_two_inequalities_the_last_term_001}
		\begin{aligned}
			& \left(f(\phi^{n})-f(\phi^{n-1}), \mathcal{A}(\phi^{n+1})\right) \\
			=&~ \frac{1}{2} \left(f(\phi^{n})-f(\phi^{n-1}), 3\phi^{n+1}-4\phi^n+\phi^{n-1}\right) \\
			=&~ \frac{3}{2} \left(f(\phi^{n})-f(\phi^{n-1}),\phi^{n+1}-\phi^n\right) - \frac{1}{2} \left(f(\phi^{n})-f(\phi^{n-1}),\phi^{n}-\phi^{n-1}\right)\\
			=&~\frac{3}{2}\left(f'(\xi_3)\left(\phi^n-\phi^{n-1}\right),{\phi}^{n+1}-\phi^n\right)-\frac{1}{2}\left(f'(\xi_3)\left(\phi^n-\phi^{n-1}\right),{\phi}^{n}-\phi^{n-1}\right),
		\end{aligned}
	\end{equation}
	for some $\xi_3$. Combining the \eqref{eq_combining_the_two_inequalities_the_last_two_term_001}-\eqref{eq_combining_the_two_inequalities_the_last_term_001}, 	
	we arrive at
	\begin{equation}\label{equality_potential}
		\begin{aligned}
			&\left(2f(\phi^{n})-f(\phi^{n-1}), \mathcal{A}(\phi^{n+1})\right) \\
			=&~ \left(F(\phi^{n+1})-F(\phi^n),1\right)+\frac{1}{2}\left(F(\phi^{n+1})-2F(\phi^n)+F(\phi^{n-1}),1\right)\\
			&-\frac{3}{4}\left(f'(\xi_1)\left(\phi^{n+1}-\phi^n\right),\phi^{n+1}-\phi^n\right)+\frac{3}{2}\left(f'(\xi_3)\left(\phi^n-\phi^{n-1}\right),{\phi}^{n+1}-\phi^n\right)  \\
			& + \frac{1}{4}\left(f'(\xi_2)\left(\phi^n-\phi^{n-1}\right),\phi^n-\phi^{n-1}\right) - \frac{1}{2}\left(f'(\xi_3)\left(\phi^n-\phi^{n-1}\right),\phi^n-\phi^{n-1}\right) .
		\end{aligned}
	\end{equation}
    Inserting \eqref{eq_Deltat_S}, \eqref{equality_potential} into \eqref{eq_combining_the_two_inequalities}, we obtain
    \begin{equation}
    	\label{eq_above_inequalities_three_00010090293}
    	\begin{aligned}
    		\Delta t & \|\nabla \mathcal{B}(\mu^{n+1})\|^2+
    		\frac{1}{4}\left(\|\nabla\phi^{n+1}\|^2+\|\nabla(2\phi^{n+1}-\phi^n)\|^2+\|\nabla(\phi^{n+1}-\phi^n)\|^2\right)\\
    		&-\frac{1}{4}\left(\|\nabla\phi^{n}\|^2+\|\nabla(2\phi^{n}-\phi^{n-1})\|^2+\|\nabla(\phi^{n}-\phi^{n-1})\|^2\right)+
    		\frac{3}{4}\|\nabla(\phi^{n+1}-2\phi^n+\phi^{n-1})\|^2\\
    		&+\frac{1}{4}\left(\|\boldsymbol{u}^{n+1}\|^2+\|2\boldsymbol{u}^{n+1}-\boldsymbol{u}^n\|^2+\|\boldsymbol{u}^{n+1}-\boldsymbol{u}^n\|^2\right)+ \Delta t\|\nabla \mathcal{B}(\boldsymbol{u}^{n+1})\|^2\\
    		&-\frac{1}{4}\left(\|\boldsymbol{u}^{n}\|^2+\|2\boldsymbol{u}^{n}-\boldsymbol{u}^{n-1}\|^2+\|\boldsymbol{u}^{n}-\boldsymbol{u}^{n-1}\|^2\right)+\frac{3}{4}\|\boldsymbol{u}^{n+1}-2\boldsymbol{u}^n+\boldsymbol{u}^{n-1}\|^2\\
    		& + \left(F(\phi^{n+1})-F(\phi^n),1\right)+\frac{1}{2}\left(F(\phi^{n+1})-2F(\phi^n)+F(\phi^{n-1}),1\right)\\
    		&+S\Delta t\|\phi^{n+1}-\phi^n\|^2 + \frac{3S\Delta t}{4}\|\phi^{n+1} - 2\phi^n + \phi^{n-1}\|^2\\
    		&+\frac{S\Delta t}{2}\left(\|\phi^{n+1}-\phi^n\|^2-\|\phi^n-\phi^{n-1}\|^2\right)\\
    		&-\frac{3}{4}\left(f'(\xi_1)\left(\phi^{n+1}-\phi^n\right),\phi^{n+1}-\phi^n\right)+\frac{3}{2}\left(f'(\xi_3)\left(\phi^n-\phi^{n-1}\right),{\phi}^{n+1}-\phi^n\right)  \\
    		& + \frac{1}{4}\left(f'(\xi_2)\left(\phi^n-\phi^{n-1}\right),\phi^n-\phi^{n-1}\right) - \frac{1}{2}\left(f'(\xi_3)\left(\phi^n-\phi^{n-1}\right),\phi^n-\phi^{n-1}\right) = 0.
    	\end{aligned}
    \end{equation}
	Using the \eqref{eq_phi_condition}, we obtain
	\begin{equation}
		\label{eq_above_inequalities_three_000100001}
		\begin{aligned}
			\Delta t & \|\nabla \mathcal{B}(\mu^{n+1})\|^2+
			\frac{1}{4}\left(\|\nabla\phi^{n+1}\|^2+\|\nabla(2\phi^{n+1}-\phi^n)\|^2+\|\nabla(\phi^{n+1}-\phi^n)\|^2\right)\\
			&-\frac{1}{4}\left(\|\nabla\phi^{n}\|^2+\|\nabla(2\phi^{n}-\phi^{n-1})\|^2+\|\nabla(\phi^{n}-\phi^{n-1})\|^2\right)+
			\frac{3}{4}\|\nabla(\phi^{n+1}-2\phi^n+\phi^{n-1})\|^2\\
			&+\frac{1}{4}\left(\|\boldsymbol{u}^{n+1}\|^2+\|2\boldsymbol{u}^{n+1}-\boldsymbol{u}^n\|^2+\|\boldsymbol{u}^{n+1}-\boldsymbol{u}^n\|^2\right)+ \Delta t\|\nabla \mathcal{B}(\boldsymbol{u}^{n+1})\|^2\\
			&-\frac{1}{4}\left(\|\boldsymbol{u}^{n}\|^2+\|2\boldsymbol{u}^{n}-\boldsymbol{u}^{n-1}\|^2+\|\boldsymbol{u}^{n}-\boldsymbol{u}^{n-1}\|^2\right)+\frac{3}{4}\|\boldsymbol{u}^{n+1}-2\boldsymbol{u}^n+\boldsymbol{u}^{n-1}\|^2\\
			&+ \left(F(\phi^{n+1})-F(\phi^n),1\right)+\frac{1}{2}\left(F(\phi^{n+1})-2F(\phi^n)+F(\phi^{n-1}),1\right)\\
			&+S\Delta t\|\phi^{n+1}-\phi^n\|^2 + \frac{3S\Delta t}{4}\|\phi^{n+1} - 2\phi^n + \phi^{n-1}\|^2\\
			&+\frac{S\Delta t}{2}\left(\|\phi^{n+1}-\phi^n\|^2-\|\phi^n-\phi^{n-1}\|^2\right)\\
			\leq & ~\frac{3L}{4} \|\phi^{n+1}-\phi^n\|^2 + \frac{3L}{2}\left({\phi}^{n+1}-\phi^n, \phi^n-\phi^{n-1}\right) + \frac{3L}{4} \|\phi^{n}-\phi^{n-1}\|^2\\
			\leq&~ \frac{3L}{2}\|\phi^{n+1}-\phi^n\|^2+\frac{3L}{2}\|\phi^n-\phi^{n-1}\|^2.
		\end{aligned}
	\end{equation}
	That is, 
	\begin{equation}
		\label{eq_above_inequalities_three_000190091}
		\begin{aligned}
			\Delta t & \|\nabla \mathcal{B}(\mu^{n+1})\|^2+
			\frac{1}{4}\left(\|\nabla\phi^{n+1}\|^2+\|\nabla(2\phi^{n+1}-\phi^n)\|^2+\|\nabla(\phi^{n+1}-\phi^n)\|^2\right)\\
			&-\frac{1}{4}\left(\|\nabla\phi^{n}\|^2+\|\nabla(2\phi^{n}-\phi^{n-1})\|^2+\|\nabla(\phi^{n}-\phi^{n-1})\|^2\right)+
			\frac{3}{4}\|\nabla(\phi^{n+1}-2\phi^n+\phi^{n-1})\|^2\\
			&+\frac{1}{4}\left(\|\boldsymbol{u}^{n+1}\|^2+\|2\boldsymbol{u}^{n+1}-\boldsymbol{u}^n\|^2+\|\boldsymbol{u}^{n+1}-\boldsymbol{u}^n\|^2\right)+ \Delta t\|\nabla \mathcal{B}(\boldsymbol{u}^{n+1})\|^2\\
			&-\frac{1}{4}\left(\|\boldsymbol{u}^{n}\|^2+\|2\boldsymbol{u}^{n}-\boldsymbol{u}^{n-1}\|^2+\|\boldsymbol{u}^{n}-\boldsymbol{u}^{n-1}\|^2\right)+\frac{3}{4}\|\boldsymbol{u}^{n+1}-2\boldsymbol{u}^n+\boldsymbol{u}^{n-1}\|^2\\
			&+ \left(F(\phi^{n+1})-F(\phi^n),1\right)+\frac{1}{2}\left(F(\phi^{n+1})-2F(\phi^n)+F(\phi^{n-1}),1\right)\\
			&+\left(S\Delta t-\frac{3L}{2}\right)\|\phi^{n+1}-\phi^n\|^2-\frac{3L}{2}\|\phi^n-\phi^{n-1}\|^2+\frac{3S\Delta t}{4}\|\phi^{n+1}-2\phi^n+\phi^{n-1}\|^2\\
			&+\frac{S\Delta t}{2}\left(\|\phi^{n+1}-\phi^n\|^2-\|\phi^n-\phi^{n-1}\|^2\right)
			\leq  0.
		\end{aligned}
	\end{equation}
	Then, if assuming 
	\begin{equation}
		S\Delta t-\frac{3L}{2}\geq \frac{3L}{2},
	\end{equation}
	that is, 
	\begin{equation}
		S\Delta t\geq 3L
		\Leftrightarrow S\geq \frac{3L}{\Delta t}.
	\end{equation}
	 We obtain the following inequality
	\begin{equation}
		\label{eq_above_inequalities_three_0001}
		\begin{aligned}
			\Delta t & \|\nabla \mathcal{B}(\mu^{n+1})\|^2+
			\frac{1}{4}\left(\|\nabla\phi^{n+1}\|^2+\|\nabla(2\phi^{n+1}-\phi^n)\|^2+\|\nabla(\phi^{n+1}-\phi^n)\|^2\right)\\
			&-\frac{1}{4}\left(\|\nabla\phi^{n}\|^2+\|\nabla(2\phi^{n}-\phi^{n-1})\|^2+\|\nabla(\phi^{n}-\phi^{n-1})\|^2\right)+
			\frac{3}{4}\|\nabla(\phi^{n+1}-2\phi^n+\phi^{n-1})\|^2\\
			&+\frac{1}{4}\left(\|\boldsymbol{u}^{n+1}\|^2+\|2\boldsymbol{u}^{n+1}-\boldsymbol{u}^n\|^2+\|\boldsymbol{u}^{n+1}-\boldsymbol{u}^n\|^2\right)+ \Delta t\|\nabla \mathcal{B}(\boldsymbol{u}^{n+1})\|^2\\
			&-\frac{1}{4}\left(\|\boldsymbol{u}^{n}\|^2+\|2\boldsymbol{u}^{n}-\boldsymbol{u}^{n-1}\|^2+\|\boldsymbol{u}^{n}-\boldsymbol{u}^{n-1}\|^2\right)+\frac{3}{4}\|\boldsymbol{u}^{n+1}-2\boldsymbol{u}^n+\boldsymbol{u}^{n-1}\|^2\\
			&+ \left(F(\phi^{n+1})-F(\phi^n),1\right)+\frac{1}{2}\left(F(\phi^{n+1})-2F(\phi^n)+F(\phi^{n-1}),1\right)\\
			& + 3L\left(\|\phi^{n+1}-\phi^n\|^2-\|\phi^n-\phi^{n-1}\|^2\right) +\frac{3S\Delta t}{4}\|\phi^{n+1}-2\phi^n+\phi^{n-1}\|^2\\
			\leq & 0.
		\end{aligned}
	\end{equation}
    Inserting the definition of energy $E^{n+1}$, we rewrite the above inequality as
	\begin{equation}
		\label{eq_combining_finally_result}
		\begin{aligned}
			&E^{n+1}-E^n + \Delta t \|\nabla \mathcal{B}(\mu^{n+1})\|^2+ \Delta t \|\nabla \mathcal{B}(\boldsymbol{u}^{n+1})\|^2
			+\frac{3}{4}\|\nabla(\phi^{n+1}-2\phi^n+\phi^{n-1})\|^2\\
			&+\frac{3}{4}\|\boldsymbol{u}^{n+1}-2\boldsymbol{u}^n+\boldsymbol{u}^{n-1}\|^2
			 + \frac{3S\Delta t}{4} \|\phi^{n+1}-2\phi^n+\phi^{n-1}\|^2
			\leq 0,
		\end{aligned}
	\end{equation} 
	where $E^{n+1}$ is defined by \eqref{eq_discrete_energy_definition}. Dropping unneeded terms on the left-hand side, then, we obtain
	\begin{equation}
		\begin{aligned}
			E^{n+1}-E^{n} \leq 0.
		\end{aligned}
	\end{equation}
	Finally, we conclude that the discrete energy is unconditionally stable.
	\end{proof}
    Thus, summing up from $n=1$ to $m$ in \eqref{eq_combining_finally_result} and using the above assumptions, we can obtain the following lemma. 
	\begin{Lemma}\label{lemma_boundness_expression}
		Assuming the initial energy is stable, i.e., $E^0 < C$ for some constant $C$. Then, for all $1<m\leq N-1$, and $S \geq {3L}/{\Delta t}$, and given the initial value of the numerical solutions $\phi^0,$ $\phi^1$, $\mu^0$, $\mu^1$, $\boldsymbol{u}^0$, $\boldsymbol{u}^1$. The solutions $\left(\phi^{n+1},\mu^{n+1},\boldsymbol{u}^{n+1},p^{n+1}\right)$ of scheme \eqref{eq_variational_form_Alg2} satisfy the following bounds
		\begin{flalign}
			\|\nabla\phi^{m+1}\|^2+\|\nabla\left(2\phi^{m+1}-\phi^m\right)\|^2+\|\nabla\left(\phi^{m+1}-\phi^m\right)\|^2\leq&~ C,\\
			\|\boldsymbol{u}^{m+1}\|^2+\|2\boldsymbol{u}^{m+1}-\boldsymbol{u}^m\|^2+\|\boldsymbol{u}^{m+1}-\boldsymbol{u}^m\|^2\leq &~C,\\
			\left(F(\phi^{m+1}),1\right)+\frac{1}{2}\left(F(\phi^{m+1})-F(\phi^m),1\right)\leq &~C,\\
			S\Delta t\sum_{n=1}^{m}\|\phi^{n+1}-2\phi^n+\phi^{n-1}\|^2\leq &~C,\\
			\sum_{n=1}^{m}\left(\|\nabla\left(\phi^{n+1}-2\phi^n+\phi^{n-1}\right)\|^2
			+\|\boldsymbol{u}^{n+1}-2\boldsymbol{u}^n+\boldsymbol{u}^{n-1}\|^2\right)\leq &~C,\\
			\Delta t\sum_{n=1}^{m}\left(\|\nabla \mathcal{B}(\mu^{n+1})\|^2+ \|\nabla \mathcal{B}(\boldsymbol{u}^{n+1})\|^2\right)\leq &~C,
		\end{flalign}
		where $C$ is the positive constant which depend on the initial values.
	\end{Lemma}
	\section{Error Estimates}\label{section_numerical_analysis}
	In the section, we analyze the error estimates for time semi-discrete scheme \eqref{eq_variational_form_Alg2}. 
	Let $\phi(t^{n+1})$, $\mu(t^{n+1})$, $\boldsymbol{u}(t^{n+1})$ and $p(t^{n+1})$
	be the exact solution of the CHNS equation at $t^{n+1}$.
	We denote the errors as follows:
	\begin{equation*}
		\label{eq_error_definitions}
		\begin{aligned}
			e^{n+1}_{\phi} &:= \phi(t^{n+1}) - \phi^{n+1}, \quad e^{n+1}_{\mu} := \mu(t^{n+1}) - \mu^{n+1}, \\
			\boldsymbol{e}^{n+1}_{\boldsymbol{u}} &:= \boldsymbol{u}(t^{n+1}) - \boldsymbol{u}^{n+1}, \quad e^{n+1}_{p} := p(t^{n+1}) - p^{n+1}.
		\end{aligned}
	\end{equation*}
	 For the optimal error estimation, we assume that the exact solution of \eqref{eq_variational_form_Alg2} satisfies the following regularity assumptions:
	 \begin{equation}
	 	\label{eq_regularity_001}
	 		\begin{aligned}
	 		&\phi\in L^{\infty}\left(0,T;H^{2}(\Omega)\cap W^{1,\infty}(\Omega)\right),  \phi_{tt}\in L^{\infty}\left(0,T;L^2(\Omega)\right) \cap L^{2}\left(0,T;H^2(\Omega)\right),  \\
	 		& \phi_{ttt}\in L^2\left(0,T;L^2(\Omega)\right),\\
	 		&\mu\in L^{\infty}\left(0,T;H^{1}(\Omega)\right), \mu_{tt}\in L^{2}\left(0,T;H^{2}(\Omega)\right),  \\
	 		&\boldsymbol{u}\in L^{\infty}\left(0,T;\boldsymbol{H}^{2}(\Omega)\right), ~ \boldsymbol{u}_{tt}\in L^{2}\left(0,T;\boldsymbol{H}^{2}(\Omega)\right) \cap L^{\infty}\left(0,T;\boldsymbol{H}^{1}(\Omega)\right),  \\
	 		& \boldsymbol{u}_{ttt}\in L^{2}\left(0,T;\boldsymbol{L}^{2}(\Omega)\right),  \\
	 		& p_{tt}\in L^{2}\left(0,T;L^{2}(\Omega)\right), ~ p_t\in L^2\left(0,T;H^1(\Omega)\right).
	 	\end{aligned}
	 \end{equation}
	The continuous variational form at moment $t^{n+1}$ is shown below: for $\forall \left(\psi, \omega, \boldsymbol{v}, q\right) \in X \times X \times \boldsymbol{Y} \times Q$, 
	\begin{equation}
		\label{eq_continuous_variational_forms}
		\begin{aligned}
			\left(\phi_t(t^{n+1}), \psi\right) + \left(\boldsymbol{u}(t^{n+1}) \cdot \nabla \phi(t^{n+1}), \psi\right) + \left(\nabla \mu(t^{n+1}), \nabla\psi\right) & =   0, \\
			- \left(\mu(t^{n+1}), \omega\right) + \left(\nabla \phi(t^{n+1}), \nabla \omega\right)  + \left(f(\phi(t^{n+1})), \omega\right)  & =   0, \\
			\left(\partial_t \boldsymbol{u}(t^{n+1}), \boldsymbol{v}\right) + \left(\nabla \boldsymbol{u}(t^{n+1}), \nabla \boldsymbol{v}\right) + \left(\boldsymbol{u}(t^{n+1}) \cdot \nabla \boldsymbol{u}(t^{n+1}), \boldsymbol{v}\right) \\
			- \left(p(t^{n+1}), \nabla\cdot \boldsymbol{v}\right) - \left(\mu(t^{n+1}) \nabla \phi(t^{n+1}), \boldsymbol{v}\right)  & =   0, \\
			\left(\nabla \cdot \boldsymbol{u}(t^{n+1}), q\right) & =  0. \\
		\end{aligned}
	\end{equation}
	The equivalent form of \eqref{eq_continuous_variational_forms} is as follows:
	\begin{equation}\label{eq_equivalent_continuous_variational_forms}
		\left\{
		\begin{aligned}
			\frac{1}{\Delta t}\left(\mathcal{A}\left(\phi(t^{n+1})\right), \psi\right) + \left(\mathcal{B}\left(\boldsymbol{u}(t^{n+1})\right) \cdot \nabla \mathcal{B}\left(\phi(t^{n+1})\right), \psi\right) +  \left(\nabla \mathcal{B}(\mu\left(t^{n+1}\right)), \nabla\psi\right) & =  \left(R_1,\psi\right),  \\
			- \left(\mathcal{B}\left(\mu(t^{n+1})\right), \omega\right) + \left(\nabla \mathcal{B}\left(\phi(t^{n+1})\right), \nabla \omega\right)  + \left(2f(\phi(t^{n}))-f(\phi(t^{n-1})), \omega\right)  & = \left(R_2,\omega\right), \\
			\frac{1}{\Delta t}\left(\mathcal{A}\left(\boldsymbol{u}(t^{n+1})\right), \boldsymbol{v}\right) + \left(\nabla \mathcal{B}\left(\boldsymbol{u}(t^{n+1})\right), \nabla \boldsymbol{v}\right) + \left(\mathcal{B}\left(\boldsymbol{u}(t^{n+1})\right) \cdot \nabla \mathcal{B}\left(\boldsymbol{u}(t^{n+1})\right), \boldsymbol{v}\right) \\
			- \left(\mathcal{B}\left(p(t^{n+1})\right), \nabla\cdot \boldsymbol{v}\right) - \left(\mathcal{B}\left(\mu(t^{n+1})\right) \nabla \mathcal{B}\left(\phi(t^{n+1})\right), \boldsymbol{v}\right)  & =  \left(R_3, \boldsymbol{v}\right),\\
			\left(\nabla \cdot \mathcal{B}\left(\boldsymbol{u}(t^{n+1})\right), q\right) & = \left(R_4,q\right), \\
		\end{aligned}\right.
	\end{equation}
	where $R_1$, $R_2$, $R_3$ and $R_4$ is defined by 
	\begin{equation}\label{eq_truncation_error}
		\left\{
		\begin{aligned}
			R_1 &:= \frac{1}{\Delta t}\mathcal{A}\left(\phi(t^{n+1})\right)- \phi_t(t^{n+1})- \Delta \left(\mathcal{B}\left(\mu\left(t^{n+1}\right)\right) - \mu(t^{n+1})\right)\\
			& \quad + \mathcal{B}\left(\boldsymbol{u}(t^{n+1})\right) \cdot \nabla \mathcal{B}\left(\phi(t^{n+1})\right) - \boldsymbol{u}(t^{n+1}) \cdot \nabla \phi(t^{n+1}),  \\
			R_2 &:= - \mathcal{B}\left(\mu(t^{n+1})\right) + \mu(t^{n+1}) -\Delta \left(\mathcal{B}\left(\phi(t^{n+1})\right) - \phi(t^{n+1})\right)\\
			& \quad + 2f(\phi(t^{n}))-f(\phi(t^{n-1}))- f(\phi(t^{n+1})),\\
			R_3 &:= \frac{1}{\Delta t}\mathcal{A}\left(\boldsymbol{u}(t^{n+1})\right) - \boldsymbol{u}_t(t^{n+1})-\Delta\left( \mathcal{B}\left(\boldsymbol{u}(t^{n+1})\right) - \boldsymbol{u}(t^{n+1})\right) \\
			& \quad + \mathcal{B}\left(\boldsymbol{u}(t^{n+1})\right) \cdot \nabla \mathcal{B}\left(\boldsymbol{u}(t^{n+1})\right) - \boldsymbol{u}(t^{n+1}) \cdot \nabla \boldsymbol{u}(t^{n+1})
			 +\nabla \left(\mathcal{B}\left(p(t^{n+1})\right) -  p(t^{n+1})\right)\\
			& \quad -\mathcal{B}\left(\mu(t^{n+1})\right) \nabla \mathcal{B}\left(\phi(t^{n+1})\right) + \mu(t^{n+1}) \nabla \phi(t^{n+1}), \\
			R_4 &:= \nabla \cdot \left(\mathcal{B}\left(\boldsymbol{u}(t^{n+1})\right) - \boldsymbol{u}(t^{n+1})\right).
		\end{aligned}\right.
	\end{equation}
	Then, by subtracting the continuous equation \eqref{eq_equivalent_continuous_variational_forms} from the discrete equation \eqref{eq_variational_form_Alg2} after setting all physical parameters to one, we obtain the following error equation, for $\forall \left(\psi, \omega, \boldsymbol{v}, q\right) \in X \times X \times \boldsymbol{Y} \times Q$,
	\begin{subequations}
		\begin{align}
		\label{eq_error_phi}
		\frac{1}{\Delta t}\left(\mathcal{A}(e^{n+1}_{\phi}), \psi\right) + \left([\mathcal{B}\left(\boldsymbol{u}(t^{n+1})\right) \cdot \nabla \mathcal{B}\left(\phi(t^{n+1})\right) - \mathcal{B}(\boldsymbol{u}^{n+1}) \cdot \nabla \bar{\phi}^{n+1}], \psi\right) \notag\\
		+  \left(\nabla \mathcal{B}(e^{n+1}_{\mu}), \nabla\psi\right) & =  \left(R_1,\psi\right),  \\
		\label{eq_error_mu}
		- \left(\mathcal{B}\left(e^{n+1}_{\mu}\right), \omega\right) + \left(\nabla \mathcal{B}\left(e^{n+1}_{\phi}\right), \nabla \omega\right) - S\Delta t\left(\tilde{\phi}^{n+1}-{\phi}^{n},\omega\right)\notag\\
		+ \left([2f(\phi(t^{n})-f(\phi(t^{n-1}))- 2f(\phi^{n})+f(\phi^{n-1})], \omega\right)  & =\left( R_2, \omega\right),\\
		\label{eq_error_ns}
		\frac{1}{\Delta t}\left(\mathcal{A}\left(e^{n+1}_{\boldsymbol{u}}\right), \boldsymbol{v}\right) + \left(\mathcal{B}\left(\boldsymbol{u}(t^{n+1})\right) \cdot \nabla \mathcal{B}\left(\boldsymbol{u}(t^{n+1})\right) - \bar{\boldsymbol{u}}^{n+1} \cdot \nabla \mathcal{B}\left(\boldsymbol{u}^{n+1}\right), \boldsymbol{v}\right) \notag\\
		+ \left(\nabla \mathcal{B}\left(e^{n+1}_{\boldsymbol{u}}\right), \nabla \boldsymbol{v}\right) - \left(\mathcal{B}\left(\mu(t^{n+1})\right) \nabla \mathcal{B}\left(\phi(t^{n+1})\right) - \mathcal{B}\left(\mu^{n+1}\right) \nabla \bar{\phi}^{n+1}, \boldsymbol{v}\right)  \notag\\
		- \left(\mathcal{B}\left(e^{n+1}_{p}\right), \nabla\cdot \boldsymbol{v}\right) & =  \left(R_3,\boldsymbol{v}\right), \\
		\label{eq_error_incompressible}
		\left(\nabla \cdot \mathcal{B}\left(e^{n+1}_{\boldsymbol{u}}\right), q\right) & = \left(R_4,q\right).
	\end{align}
	\end{subequations}
	We firstly need the following lemmas which will be used in the later analysis.
	\begin{Lemma}[\cite{2020_Aytekin_Cibik_Analysis_of_second_order_time_filtered_backward_Euler_method_for_MHD_equations}]
		\label{lemma_consistency_error}
		Let $w, w_t, w_{t t}, w_{t t t} \in L^2\left(0, T; L^2(\Omega)\right)$, then there exists $C>0$ such that
		\begin{equation*}
			\begin{aligned}
				\Delta t \sum_{n=0}^{N-1}\left\|\mathcal{B}\left(w(t^{n+1})\right) - w(t^{n+1})\right\|^2 &\leq C \left(\Delta t\right)^4\left\|w_{t t}\right\|_{L^2\left(0, T ; L^2(\Omega)\right)}^2, \\
				\Delta t \sum_{n=0}^{N-1}\left\|\frac{1}{\Delta t}\mathcal{A}\left(w(t^{n+1})\right) - w_t(t^{n+1})\right\|^2 &\leq C \left(\Delta t\right)^4\left\|w_{t t t}\right\|_{L^2\left(0, T ; L^2(\Omega)\right)}^2.
			\end{aligned}
		\end{equation*}
	\end{Lemma}
    \noindent The following two lemmas' proofs can be found in the Appendix \ref{Appendix-Lemma-4.2} and Appendix \ref{Appendix-Lemma-4.3}.
    \begin{Lemma}\label{Lemma_Estimation_Truncation_error}
    	For the truncation errors defined in \eqref{eq_truncation_error}, we have
    	\begin{equation*}
    		\begin{aligned}
    			\| R_1 \| &\leq C  \left\|\frac{1}{\Delta t}\mathcal{A}\left(\phi(t^{n+1})\right)- \phi_t(t^{n+1})\right\| + C \|\Delta \left(\mathcal{B}\left(\mu\left(t^{n+1}\right) \right)- \mu(t^{n+1})\right)\|   \\
    			& \quad + C \|\mathcal{B}\left(\boldsymbol{u}(t^{n+1})\right) - \boldsymbol{u}(t^{n+1})\|_1 \cdot \|\mathcal{B}\left(\phi(t^{n+1})\right)\|_2  \\
    			& \quad + C \|\boldsymbol{u}(t^{n+1})\|_1 \cdot \|\phi(t^{n+1}) - \mathcal{B}\left(\phi(t^{n+1})\right)\|_2.  \\
    			\| \nabla R_2 \| &\leq C \|\nabla \left(\mathcal{B}(\mu(t^{n+1}))-\mu(t^{n+1})\right)\| + C \|\nabla \Delta \left(\mathcal{B}(\phi(t^{n+1}))-\phi(t^{n+1})\right)\| \\
    			& \quad  + CL\left(\Delta t\right)^2 \|\phi_{tt}\|_{L^{\infty}(0,T;H^1)}.  \\
    			\| R_3 \| &\leq C\left\|\frac{1}{\Delta t}\mathcal{A}(\boldsymbol{u}(t^{n+1}))-\boldsymbol{u}_t(t^{n+1})\right\| + C\left\|\Delta\left(\mathcal{B}(\boldsymbol{u}(t^{n+1}))-\boldsymbol{u}(t^{n+1})\right)\right\|  \\
    			& \quad + C\|\mathcal{B}\left(\boldsymbol{u}(t^{n+1})\right)\|_1 \cdot \left\|\mathcal{B}\left(\boldsymbol{u}(t^{n+1})\right) - \boldsymbol{u}(t^{n+1})\right\|_2  \\
    			& \quad + C\|\boldsymbol{u}(t^{n+1})\|_2 \cdot \left\|\mathcal{B}\left(\boldsymbol{u}(t^{n+1})\right) - \boldsymbol{u}(t^{n+1})\right\|_1  \\
    			& \quad + C\|\mathcal{B}\left(\mu(t^{n+1})\right)\|_1 \cdot \left\|\mathcal{B}\left(\phi(t^{n+1})\right) - \phi(t^{n+1}) \right\|_2  \\
    			& \quad + C\|\phi(t^{n+1})\|_2 \cdot \left\|\mathcal{B}\left(\mu(t^{n+1})\right) - \mu(t^{n+1})\right\|_1  \\
    			& \quad + C \|\nabla\mathcal{B}(p(t^{n+1}))-p(t^{n+1}))\|.
    		\end{aligned}
    	\end{equation*}
    \end{Lemma}
    \begin{Lemma}\label{Lemma_A_e_phi_dt}
    	We can bound as
    	\begin{equation}\label{eq_A_e_phi_dt}
    		\begin{aligned}
    			\left\|\frac{\mathcal{A}(e_{\phi}^{n+1})}{\Delta t}\right\|_{-1} &\leq C\Bigg\|\mathcal{B}(\phi(t^{n+1}))\Bigg\|_2 \cdot \Bigg\|\mathcal{B}(e_{\boldsymbol{u}}^{n+1})\Bigg\| + C\Bigg\|\nabla\mathcal{B}(\boldsymbol{u}^{n+1})\Bigg\| \cdot  \Bigg\|\nabla\left(2e_{\phi}^n-e_{\phi}^{n-1}\right)\Bigg\|  \\
    			& \quad + C\left(\Delta t\right)^2 \Bigg\|\nabla \mathcal{B}(\boldsymbol{u}^{n+1})\Bigg\| \cdot\Bigg\|\phi_{tt}\Bigg\|_{L^{\infty}(0,T;H^1)} + \Bigg\|\nabla \mathcal{B}(e^{n+1}_{\mu})\Bigg\| + C\Bigg\|R_1\Bigg\|.
    		\end{aligned}
    	\end{equation}
    \end{Lemma}
	Next, we can obtain the following lemma.
	\begin{Lemma}\label{lemma_main_results}
		We can obtain
		\begin{equation}\label{eq_main_results}
			\begin{aligned}
				\frac{1}{4 \Delta t}&\bigg(\|\nabla e_{\phi}^{n+1}\|^2+\|\nabla \left(2e_{\phi}^{n+1}-e_{\phi}^n\right)\|^2
				+\|\nabla \left(e_{\phi}^{n+1}-e_{\phi}^n\right)\|^2  \\ 
				& \qquad  +\|e_{\boldsymbol{u}}^{n+1}\|^2+\|2e_{\boldsymbol{u}}^{n+1}-e_{\boldsymbol{u}}^n\|^2
				+\|e_{\boldsymbol{u}}^{n+1}-e_{\boldsymbol{u}}^n\|^2\bigg)\\
				&-\frac{1}{4 \Delta t}\bigg(\|\nabla e_{\phi}^{n}\|^2+\|\nabla \left(2e_{\phi}^{n}-e_{\phi}^{n-1}\right)\|^2
				+\|\nabla \left(e_{\phi}^{n}-e_{\phi}^{n-1}\right)\|^2  \\  
				& \qquad +\|e_{\boldsymbol{u}}^{n}\|^2+\|2e_{\boldsymbol{u}}^{n}-e_{\boldsymbol{u}}^{n-1}\|^2
				+\|e_{\boldsymbol{u}}^{n}-e_{\boldsymbol{u}}^{n-1}\|^2\bigg)\\
				& + \frac12\|\nabla\mathcal{B}(e_{\mu}^{n+1})\|^2 + \frac12\|\nabla\mathcal{B}(e_{\boldsymbol{u}}^{n+1})\|^2 + S \|\mathcal{A}(e_{\phi}^{n+1})\|^2\\
				& +\frac{3}{4 \Delta t}\left(\|\nabla \left(e_{\phi}^{n+1}-2e_{\phi}^n+e_{\phi}^{n-1}\right)\|^2
				+\| e_{\boldsymbol{u}}^{n+1}-2e_{\boldsymbol{u}}^n+e_{\boldsymbol{u}}^{n-1}\|^2\right)\\
				\leq ~ & C \|R_1\|^2 + C \|\nabla R_2\|^2 + C \|R_3\|^2 + \epsilon_{\phi}\left\| \frac{\mathcal{A}(e_{\phi}^{n+1})}{\Delta t} \right\|^2_{-1} + C_{1} (\Delta t)^4 \\
				& + C_u \left(\|\mathcal{B}(e^{n+1}_{\boldsymbol{u}})\|^2 + \|2e^{n}_{\boldsymbol{u}} - e^{n-1}_{\boldsymbol{u}} \|^2\right) \\
				& + C_{\phi}\left(\|\nabla e_{\phi}^{n}\|^2+\|\nabla\left(2e_{\phi}^{n}-e_{\phi}^{n-1}\right)\|^2+\|\nabla\left(e_{\phi}^{n}-e_{\phi}^{m-1}\right)\|^2\right),
			\end{aligned}
		\end{equation}
	    where $\epsilon_{\phi}$ is some positive constant, and 
	    \begin{equation}\label{eq_three_symobols}
	    	\begin{aligned}
	    		C_{1} &:= C\left(\|\nabla \mathcal{B}(\boldsymbol{u}^{n+1})\|^2 \cdot\|\phi_{tt}\|^2_{L^{\infty}(0,T;H^1)} + S^2 \|\phi_t\|_{L^{\infty}(0,T;H^1)} \right.  \\
	    		& \left. \quad + \|\nabla\mathcal{B}(\boldsymbol{u}^{n+1})\|^2\cdot \left\|\boldsymbol{u}_{tt} \right\|^2_{L^{\infty}(0,T;H^1)} + \|\nabla\mathcal{B}(\mu^{n+1})\|^2\cdot\|\phi_{tt}\|_{L^{\infty}(0,T;H^1)}^2 \right), \\
	    		C_u &:= C \max\left\{\|\mathcal{B}(\phi(t^{n+1}))\|^2_2, \; \| \mathcal{B}(\boldsymbol{u}(t^{n+1}))\|^2_2\right\} , \\
	    		C_{\phi} &:= C \max\left\{\frac{36L^2}{\epsilon_{\phi}}, \; \frac{12L^2}{\epsilon_{\phi}} + \|\nabla \mathcal{B}(\boldsymbol{u}^{n+1})\|^2 + \|\nabla \mathcal{B}(\mu^{n+1})\|^2\right\}.
	    	\end{aligned}
	    \end{equation}
	\end{Lemma}
	\begin{proof}
		Setting $\psi=\mathcal{B}(e_{\mu}^{n+1})$ in \eqref{eq_error_phi}, $\omega=\mathcal{A}(e_{\phi}^{n+1})/\Delta t$ in \eqref{eq_error_mu},
		$\boldsymbol{v}=\mathcal{B}(e_{\boldsymbol{u}}^{n+1})$ in \eqref{eq_error_ns}, $q=\mathcal{B}(e_p^{n+1})$ in \eqref{eq_error_incompressible}, and using the fact that $\nabla\cdot\boldsymbol{u}(t)=0$ for $t = t_{n+1}, t_n, t_{n-1}$ to get the vanishing $R_4$, then using \eqref{eq_A_B_form}, we have
		\begin{equation}\label{eq_error_phi_inner_product_e_phi}
			\begin{aligned}
				\frac{1}{4 \Delta t}&\bigg(\|\nabla e_{\phi}^{n+1}\|^2+\|\nabla \left(2e_{\phi}^{n+1}-e_{\phi}^n\right)\|^2
				+\|\nabla \left(e_{\phi}^{n+1}-e_{\phi}^n\right)\|^2  \\ 
				& \qquad  +\|e_{\boldsymbol{u}}^{n+1}\|^2+\|2e_{\boldsymbol{u}}^{n+1}-e_{\boldsymbol{u}}^n\|^2
				+\|e_{\boldsymbol{u}}^{n+1}-e_{\boldsymbol{u}}^n\|^2\bigg)\\
				&-\frac{1}{4 \Delta t}\bigg(\|\nabla e_{\phi}^{n}\|^2+\|\nabla \left(2e_{\phi}^{n}-e_{\phi}^{n-1}\right)\|^2
				+\|\nabla \left(e_{\phi}^{n}-e_{\phi}^{n-1}\right)\|^2  \\  
				& \qquad +\|e_{\boldsymbol{u}}^{n}\|^2+\|2e_{\boldsymbol{u}}^{n}-e_{\boldsymbol{u}}^{n-1}\|^2
				+\|e_{\boldsymbol{u}}^{n}-e_{\boldsymbol{u}}^{n-1}\|^2\bigg)\\
				& + \|\nabla\mathcal{B}(e_{\mu}^{n+1})\|^2 + \|\nabla\mathcal{B}(e_{\boldsymbol{u}}^{n+1})\|^2 \\
				& +\frac{3}{4 \Delta t}\left(\|\nabla \left(e_{\phi}^{n+1}-2e_{\phi}^n+e_{\phi}^{n-1}\right)\|^2
				+\| e_{\boldsymbol{u}}^{n+1}-2e_{\boldsymbol{u}}^n+e_{\boldsymbol{u}}^{n-1}\|^2\right)\\
				=& \left(R_1,\mathcal{B}(e_{\mu}^{n+1})\right) - \left(\mathcal{B}(\boldsymbol{u}(t^{n+1})) \cdot \nabla \mathcal{B}\left(\phi(t^{n+1})\right) - \mathcal{B}(\boldsymbol{u}^{n+1}) \cdot \nabla \bar{\phi}^{n+1}, \mathcal{B}(e_{\mu}^{n+1})\right)\\
				& - \left(2f\left(\phi(t^{n+1})\right)-f(\phi(t^{n}))-\left(2f(\phi^{n})-f(\phi^{n-1})\right), \frac{\mathcal{A}(e_{\phi}^{n+1})}{\Delta t}\right)  \\
				& + S\Delta t \left(\tilde{\phi}^{n+1}-{\phi}^{n},\frac{\mathcal{A}(e_{\phi}^{n+1})}{\Delta t}\right)  + \left(R_2,\frac{\mathcal{A}(e_{\phi}^{n+1})}{\Delta t}\right) + \left(R_3,\mathcal{B}(e_{\boldsymbol{u}}^{n+1})\right)\\
				& - \left(\mathcal{B}(\boldsymbol{u}(t^{n+1}))\cdot\nabla\mathcal{B}(\boldsymbol{u}(t^{n+1}))-\bar{\boldsymbol{u}}^{n+1}\cdot\nabla\mathcal{B}(\boldsymbol{u}^{n+1}),\mathcal{B}(e_{\boldsymbol{u}}^{n+1})\right)\\
				& + \left(\mathcal{B}(\mu(t^{n+1}))\nabla\mathcal{B}(\phi(t^{n+1}))-\mathcal{B}(\mu^{n+1})\nabla\bar{\phi}^{n+1},\mathcal{B}(e_{\boldsymbol{u}}^{n+1})\right).
			\end{aligned}
		\end{equation}
	We bound the terms on the RHS of \eqref{eq_error_phi_inner_product_e_phi}, term by term as follows. Using the definition of $R_1$ in \eqref{eq_truncation_error},  we have 
	\begin{equation}
		\label{eq_first_terms_bound}
		\begin{aligned}
			\bigg|\left(R_1,\mathcal{B}(e_{\mu}^{n+1})\right)\bigg| \leq C \|R_1\|^2 + \frac{\epsilon_{\mu}}{3}  \|\nabla \mathcal{B}(e_{\mu}^{n+1})\|^2 .
		\end{aligned}
	\end{equation}
	The second term on the right-hand side of \eqref{eq_error_phi_inner_product_e_phi} can be estimated by
	\begin{equation}\label{eq_second_term_bound}
		\begin{aligned}
			&~\bigg|\left(\mathcal{B}(\boldsymbol{u}(t^{n+1})) \cdot \nabla \mathcal{B}\left(\phi(t^{n+1})\right) - \mathcal{B}(\boldsymbol{u}^{n+1}) \cdot \nabla \bar{\phi}^{n+1}, \mathcal{B}(e_{\mu}^{n+1})\right)\bigg|\\
			= &~ \bigg|\left(\mathcal{B}(e_{\boldsymbol{u}}^{n+1})\nabla\mathcal{B}(\phi(t^{n+1}))+\mathcal{B}(\boldsymbol{u}^{n+1})\nabla\left(\mathcal{B}(\phi(t^{n+1}))-2\phi(t^{n})+\phi(t^{n-1})\right), \mathcal{B}(e_{\mu}^{n+1})\right)\\
			&~+\left(\mathcal{B}(\boldsymbol{u}^{n+1})\nabla\left(2e_{\phi}^n-e_{\phi}^{n-1}\right),\mathcal{B}(e_{\mu}^{n+1})\right)\bigg|\\
			\leq & ~ C \|\mathcal{B}(\phi(t^{n+1}))\|_2 \|\mathcal{B}(e_{\boldsymbol{u}}^{n+1})\| \|\nabla\mathcal{B}(e_{\mu}^{n+1})\| + C \|\nabla\mathcal{B}(\boldsymbol{u}^{n+1})\| \|\nabla\left(2e_{\phi}^n-e_{\phi}^{n-1}\right)\| \|\nabla \mathcal{B}(e_{\mu}^{n+1})\| \\
			& + C \|\nabla\mathcal{B}(\boldsymbol{u}^{n+1})\| \|\nabla\left(\mathcal{B}(\phi(t^{n+1}))-2\phi(t^{n})+\phi(t^{n-1})\right)\| \|\nabla\mathcal{B}(e_{\mu}^{n+1})\|  \\
			\leq & ~ \frac{\epsilon_{\mu}}{3} \|\nabla\mathcal{B}(e_{\mu}^{n+1})\|^2 + C \|\mathcal{B}(\phi(t^{n+1}))\|^2_2 \cdot\|\mathcal{B}(e_{\boldsymbol{u}}^{n+1})\|^2 + C \left(\Delta t\right)^4 \|\nabla \mathcal{B}(\boldsymbol{u}^{n+1})\|^2 \cdot\|\phi_{tt}\|^2_{L^{\infty}(0,T;H^1)} \\
			& + C \|\nabla \mathcal{B}(\boldsymbol{u}^{n+1})\|^2 \cdot \|\nabla \left(2e_{\phi}^n-e_{\phi}^{n-1}\right)\|^2. 
		\end{aligned}
	\end{equation}
	Using the \eqref{eq_phi_condition}, the third term on the right-hand side of \eqref{eq_error_phi_inner_product_e_phi} can be estimated by
	\begin{equation}
		\label{eq_third_term_bound}
		\begin{aligned}
			&\bigg|\left(2f\left(\phi(t^{n})\right)-f(\phi(t^{n-1}))-\left(2f(\phi^{n})-f(\phi^{n-1})\right), \frac{\mathcal{A}(e_{\phi}^{n+1})}{\Delta t}\right)\bigg|\\
			\leq & \left\|\nabla \left(2f\left(\phi(t^{n})\right)-f(\phi(t^{n-1}))-\left(2f(\phi^{n})-f(\phi^{n-1})\right)\right) \right\|\cdot \left\| \frac{\mathcal{A}(e_{\phi}^{n+1})}{\Delta t} \right\|_{-1}  \\
			\leq & \frac{\epsilon_{\phi}}{3} \left\| \frac{\mathcal{A}(e_{\phi}^{n+1})}{\Delta t} \right\|^2_{-1} + \frac{3}{\epsilon_{\phi}} \|2f^{\prime}\left(\xi_1\right) \nabla e^n_{\phi} - f^{\prime}\left(\xi_2\right) \nabla e^{n-1}_{\phi}\|^2 \\
			\leq & \frac{\epsilon_{\phi}}{3} \left\| \frac{\mathcal{A}(e_{\phi}^{n+1})}{\Delta t} \right\|^2_{-1} + \frac{6L^2}{\epsilon_{\phi}} \left(\|2\nabla e^n_{\phi}\|^2 + \|\nabla e^{n-1}_{\phi}\|^2\right) \\
			\leq & \frac{\epsilon_{\phi}}{3} \left\| \frac{\mathcal{A}(e_{\phi}^{n+1})}{\Delta t} \right\|^2_{-1} + \frac{6L^2}{\epsilon_{\phi}} \left(2\left\|\nabla \left(2e^n_{\phi} - e^{n-1}_{\phi}\right)\right\|^2 + 6\left\|\nabla \left(e^n_{\phi} - e^{n-1}_{\phi}\right)\right\|^2 + 6\left\|\nabla e^n_{\phi}\right\|^2\right).
		\end{aligned}
	\end{equation}
	The fourth and fifth terms on the right-hand side of \eqref{eq_error_phi_inner_product_e_phi} can be bounded by
	\begin{equation}
		\label{eq_fourth_term_bound}
		\begin{aligned}
			&S\Delta t\left(\tilde{\phi}^{n+1}-{\phi}^{n}, \frac{\mathcal{A}(e_{\phi}^{n+1})}{\Delta t}\right) = S\Delta t\left(\mathcal{A}(\phi^{n+1}), \frac{\mathcal{A}(e_{\phi}^{n+1})}{\Delta t}\right) \\
			=~ & S\Delta t\left(\mathcal{A}(\phi(t^{n+1})) - \mathcal{A}(e_{\phi}^{n+1}), \frac{\mathcal{A}(e_{\phi}^{n+1})}{\Delta t}\right)  \\
			=~ & - S \|\mathcal{A}(e_{\phi}^{n+1})\|^2 + S\Delta t\left(\mathcal{A}(\phi(t^{n+1})), \frac{\mathcal{A}(e_{\phi}^{n+1})}{\Delta t}\right)  \\
			\leq & ~ - S \|\mathcal{A}(e_{\phi}^{n+1})\|^2 + S\Delta t \|\nabla \mathcal{A}(\phi(t^{n+1}))\| \cdot \left\|\frac{\mathcal{A}(e_{\phi}^{n+1})}{\Delta t}\right\|_{-1}  \\
			\leq & ~ - S \|\mathcal{A}(e_{\phi}^{n+1})\|^2 + \frac{\epsilon_{\phi}}{3} \left\| \frac{\mathcal{A}(e_{\phi}^{n+1})}{\Delta t} \right\|^2_{-1} + CS^2(\Delta t)^4 \|\phi_t\|^2_{L^{\infty}(0,T;H^1)}.
		\end{aligned}
	\end{equation}
	and 
	\begin{equation}
		\label{eq_fifth_term_bound}
		\begin{aligned}
			\bigg|\left(R_2, \frac{\mathcal{A}(e_{\phi}^{n+1})}{\Delta t}\right)\bigg| \leq C\| \nabla R_2\|^2 + \frac{\epsilon_{\phi}}{3} \left\| \frac{\mathcal{A}(e_{\phi}^{n+1})}{\Delta t} \right\|^2_{-1}.
		\end{aligned}
	\end{equation}
	The sixth term on the right-hand side of \eqref{eq_error_phi_inner_product_e_phi} can be bounded by
	\begin{equation}
		\label{eq_sixth_term_bound}
		\begin{aligned}
			\left|\left(R_3,\mathcal{B}(e_{\boldsymbol{u}}^{n+1})\right)\right| \leq C\|R_3\|^2 + \frac{\epsilon_{\boldsymbol{u}}}{3} \|\nabla\mathcal{B}(e_{\boldsymbol{u}}^{n+1})\|^2.
		\end{aligned}
	\end{equation}
	The seventh term on the right-hand side of \eqref{eq_error_phi_inner_product_e_phi} can be bounded by
	\begin{equation}
		\label{eq_seventh_term_bound}
		\begin{aligned}
			& \bigg|-\left(\mathcal{B}(\boldsymbol{u}(t^{n+1}))\cdot\nabla\mathcal{B}(\boldsymbol{u}(t^{n+1}))-\bar{\boldsymbol{u}}^{n+1}\cdot\nabla\mathcal{B}(\boldsymbol{u}^{n+1}),\mathcal{B}(e_{\boldsymbol{u}}^{n+1})\right)\bigg|\\
			= & \bigg|\left(\mathcal{B}(\boldsymbol{u}(t^{n+1}))\cdot\nabla\mathcal{B}(\boldsymbol{u}(t^{n+1})-\boldsymbol{u}^{n+1})+\left(\mathcal{B}(\boldsymbol{u}(t^{n+1}))-\bar{\boldsymbol{u}}^{n+1}\right)\cdot\nabla\mathcal{B}(\boldsymbol{u}^{n+1}),\mathcal{B}(e_{\boldsymbol{u}}^{n+1})\right)\bigg|\\
			\leq& \bigg|\left(\mathcal{B}(\boldsymbol{u}(t^{n+1}))\cdot\nabla\mathcal{B}(e_{\boldsymbol{u}}^{n+1}),\mathcal{B}(e_{\boldsymbol{u}}^{n+1})\right)\bigg| + \bigg|\left(\left(2e_{\boldsymbol{u}}^{n}-e_{\boldsymbol{u}}^{n-1}\right)\cdot\nabla\mathcal{B}(\boldsymbol{u}(t^{n+1})),\mathcal{B}(e_{\boldsymbol{u}}^{n+1})\right)\bigg|\\
			&   +  \bigg|\left(\left(\mathcal{B}(\boldsymbol{u}(t^{n+1}))-2\boldsymbol{u}(t^{n})+\boldsymbol{u}(t^{n-1})\right)\cdot\nabla\mathcal{B}(\boldsymbol{u}^{n+1}),\mathcal{B}(e_{\boldsymbol{u}}^{n+1})\right)\bigg|  \\
			\leq &  C\|\mathcal{B}(\boldsymbol{u}(t^{n+1}))\|^2_2 \cdot\|2e_{\boldsymbol{u}}^{n}-e_{\boldsymbol{u}}^{n-1}\|^2   
			+ \frac{\epsilon_{\boldsymbol{u}}}{3} \|\nabla\mathcal{B}(e_{\boldsymbol{u}}^{n+1})\|^2   \\ 
			&   + C  \|\nabla\mathcal{B}(\boldsymbol{u}^{n+1})\|^2 \cdot \left\|\nabla \left(\mathcal{B}(\boldsymbol{u}(t^{n+1}))-2\boldsymbol{u}(t^{n})+\boldsymbol{u}(t^{n-1})\right) \right\|^2   \\
			\leq & C  \| \mathcal{B}(\boldsymbol{u}(t^{n+1}))\|^2_2 \cdot\|2e_{\boldsymbol{u}}^{n}-e_{\boldsymbol{u}}^{n-1}\|^2 + \frac{\epsilon_{\boldsymbol{u}}}{3}  \|\nabla\mathcal{B}(e_{\boldsymbol{u}}^{n+1})\|^2 \\
			& +  C\left(\Delta t\right)^4 \|\nabla\mathcal{B}(\boldsymbol{u}^{n+1})\|^2\cdot \left\|\boldsymbol{u}_{tt} \right\|^2_{L^{\infty}(0,T;H^1)}. 
		\end{aligned}
	\end{equation}
	The eighth term on the right-hand side of \eqref{eq_error_phi_inner_product_e_phi} can be bounded by
	\begin{equation}
		\label{eq_eighth_term_bound}
		\begin{aligned}
			&\bigg|\left(\mathcal{B}(\mu(t^{n+1}))\nabla\mathcal{B}(\phi(t^{n+1}))-\mathcal{B}(\mu^{n+1})\nabla\bar{\phi}^{n+1},\mathcal{B}(e_{\boldsymbol{u}}^{n+1})\right)\bigg|\\
			=& \bigg| \left(\mathcal{B}(\mu(t^{n+1})-\mu^{n+1})\nabla \mathcal{B}(\phi(t^{n+1})),\mathcal{B}(e_{\boldsymbol{u}}^{n+1})\right)  \\
			&  +\left(\mathcal{B}(\mu^{n+1})\nabla\left(2e_{\phi}^n-e_{\phi}^{n-1}\right),\mathcal{B}(e_{\boldsymbol{u}}^{n+1})\right) \\
			& + \left(\mathcal{B}(\mu^{n+1})\nabla\left(\mathcal{B}(\phi(t^{n+1}))-2\phi(t^{n})+\phi(t^{n-1})\right),\mathcal{B}(e_{\boldsymbol{u}}^{n+1})\right) \bigg|  \\
			\leq&~  C\|\nabla \mathcal{B}(e_{\mu}^{n+1})\|\cdot\|\mathcal{B}(\phi(t^{n+1}))\|_{2} \cdot\| \mathcal{B}(e_{\boldsymbol{u}}^{n+1})\| \\
			& + C\|\nabla \mathcal{B}(\mu^{n+1})\|\cdot \left\|\nabla \left(2e_{\phi}^{n}-e_{\phi}^{n-1}\right)\right\|\cdot\|\nabla \mathcal{B}(e_{\boldsymbol{u}}^{n+1})\|  \\
			& +  C \|\nabla \mathcal{B}(\mu^{n+1})\|\cdot\|\nabla\left(\phi(t^{n+1})-2\phi(t^n)+\phi(t^{n-1})\right)\|\cdot\|\nabla \mathcal{B}(e_{\boldsymbol{u}}^{n+1})\| \\
			\leq&~  C \|\mathcal{B}(\phi(t^{n+1}))\|^2_{2}\cdot\|\mathcal{B}(e_{\boldsymbol{u}}^{n+1})\|^2 + \frac{\epsilon_{\mu}}{3} \|\nabla \mathcal{B}(e_{\mu}^{n+1})\|^2 \\
			& + C \|\nabla \mathcal{B}(\mu^{n+1})\|^2\cdot \left\|\nabla \left(2e_{\phi}^{n}-e_{\phi}^{n-1}\right)\right\|^2 \\
			&  + \frac{\epsilon_{\boldsymbol{u}}}{3} \|\nabla\mathcal{B}(e_{\boldsymbol{u}}^{n+1})\|^2 
			+ C\left(\Delta t\right)^4\|\nabla\mathcal{B}(\mu^{n+1})\|^2\cdot\|\phi_{tt}\|_{L^{\infty}(0,T;H^1)}^2.
		\end{aligned}
	\end{equation}
    Combining the above inequalities \eqref{eq_first_terms_bound}-\eqref{eq_eighth_term_bound} with \eqref{eq_error_phi_inner_product_e_phi}, and choosing $\epsilon_{\mu} = \epsilon_{\boldsymbol{u}} = 1/2$ can complete this proof.
	\end{proof}

    Finally, we obtain the error estimation of scheme \eqref{eq_variational_form_Alg2} in the following theorem.
    \begin{Theorem}\label{theorem_main_results}
    	Supposing the regularity assumptions \eqref{eq_regularity_001} hold, and assuming $\Delta t$ small enough such that $C^*\Delta t \leq 1$, where 
    	$$
    	\begin{aligned}
    		C^* :=& \max\left\{C \max\left\{\|\mathcal{B}(\phi(t^{n+1}))\|^2_2, \; \|\mathcal{B}(\boldsymbol{u}(t^{n+1}))\|^2_2\right\}, \right. \\
    		& \qquad\; \left. C \max\left\{3L^2, \; L^2 + \|\nabla \mathcal{B}(\boldsymbol{u}^{n+1})\|^2 + \|\nabla \mathcal{B}(\mu^{n+1})\|^2\right\}\right\}.
    	\end{aligned}
    	$$
    	Then, for all $1\leq m\leq N-1$, there exists a positive constant $C$ such that the solution of scheme \eqref{eq_variational_form_Alg2} fulfills error estimations
    	\begin{equation}
    		\begin{aligned}
    			\|\nabla e_{\phi}^{m+1}\|^2&+\|\nabla\left(2e_{\phi}^{m+1}-e_{\phi}^n\right)\|^2+\|\nabla\left(e_{\phi}^{m+1}-e_{\phi}^{m}\right)\|^2  \\
    			& \quad 
    			+\|e_{\boldsymbol{u}}^{m+1}\|^2+\|2e_{\boldsymbol{u}}^{m+1}-e_{\boldsymbol{u}}^m\|^2+\|e_{\boldsymbol{u}}^{m+1}-e_{\boldsymbol{u}}^{m}\|^2\\
    			&+\Delta t\sum_{n=1}^{m}\|\nabla\mathcal{B}(e_{\mu}^{n+1})\|^2 + \Delta t\sum_{n=1}^{m}\|\nabla\mathcal{B}(e_{\boldsymbol{u}}^{n+1})\|^2 + \Delta t\sum_{n=1}^{m}\|\nabla\mathcal{B}(e_{\phi}^{n+1})\|^2  \\
    			\leq & C\Big(\left(\Delta t\right)^4 + \|\nabla e_{\phi}^{1}\|^2 + \|\nabla (2e_{\phi}^{1}-e_{\phi}^{0})\|^2
    			+\|\nabla (e_{\phi}^{n}-e_{\phi}^{0})\|^2  \\
    			& + \|e_{\boldsymbol{u}}^{1}\|^2+\|2e_{\boldsymbol{u}}^{1}-e_{\boldsymbol{u}}^{0}\|^2
    			+\|e_{\boldsymbol{u}}^{1}-e_{\boldsymbol{u}}^{0}\|^2\Big).
    		\end{aligned}
    	\end{equation}
    \end{Theorem}
    \begin{proof}
    	Inserting \eqref{eq_A_e_phi_dt} in the Lemma \ref{Lemma_A_e_phi_dt} into the \eqref{eq_main_results}  in the Lemma \ref{lemma_main_results} to get
    	\begin{equation}
    		\begin{aligned}
    			\frac{1}{4 \Delta t}&\bigg(\|\nabla e_{\phi}^{n+1}\|^2+\|\nabla \left(2e_{\phi}^{n+1}-e_{\phi}^n\right)\|^2
    			+\|\nabla \left(e_{\phi}^{n+1}-e_{\phi}^n\right)\|^2  \\ 
    			& \qquad  +\|e_{\boldsymbol{u}}^{n+1}\|^2+\|2e_{\boldsymbol{u}}^{n+1}-e_{\boldsymbol{u}}^n\|^2
    			+\|e_{\boldsymbol{u}}^{n+1}-e_{\boldsymbol{u}}^n\|^2\bigg)\\
    			&-\frac{1}{4 \Delta t}\bigg(\|\nabla e_{\phi}^{n}\|^2+\|\nabla \left(2e_{\phi}^{n}-e_{\phi}^{n-1}\right)\|^2
    			+\|\nabla \left(e_{\phi}^{n}-e_{\phi}^{n-1}\right)\|^2  \\  
    			& \qquad +\|e_{\boldsymbol{u}}^{n}\|^2+\|2e_{\boldsymbol{u}}^{n}-e_{\boldsymbol{u}}^{n-1}\|^2
    			+\|e_{\boldsymbol{u}}^{n}-e_{\boldsymbol{u}}^{n-1}\|^2\bigg)\\
    			& + \frac12\|\nabla\mathcal{B}(e_{\mu}^{n+1})\|^2 + \frac12\|\nabla\mathcal{B}(e_{\boldsymbol{u}}^{n+1})\|^2 + S \|\mathcal{A}(e_{\phi}^{n+1})\|^2\\
    			& +\frac{3}{4 \Delta t}\left(\|\nabla \left(e_{\phi}^{n+1}-2e_{\phi}^n+e_{\phi}^{n-1}\right)\|^2
    			+\| e_{\boldsymbol{u}}^{n+1}-2e_{\boldsymbol{u}}^n+e_{\boldsymbol{u}}^{n-1}\|^2\right)\\
    			\leq ~ & C \|R_1\|^2 + C \|\nabla R_2\|^2 + C \|R_3\|^2 + \epsilon_{\phi}\left\|\nabla \mathcal{B}(e^{n+1}_{\mu})\right\| ^2 + C_{1} (\Delta t)^4 \\
    			& + \tilde{C}_u \left(\|\mathcal{B}(e^{n+1}_{\boldsymbol{u}})\|^2 + \|2e^{n}_{\boldsymbol{u}} - e^{n-1}_{\boldsymbol{u}} \|^2\right) \\
    			& + \tilde{C}_{\phi}\left(\|\nabla e_{\phi}^{n}\|^2+\|\nabla\left(2e_{\phi}^{n}-e_{\phi}^{n-1}\right)\|^2+\|\nabla\left(e_{\phi}^{n}-e_{\phi}^{m-1}\right)\|^2\right),
    		\end{aligned}
    	\end{equation}
    	where $C_{1}$ is defined in \eqref{eq_three_symobols}, $\epsilon_{\phi}$ is some positive constant, and 
    	\begin{equation}
    		\begin{aligned}
    			\tilde{C}_u &:= C \max\left\{2\|\mathcal{B}(\phi(t^{n+1}))\|^2_2, \; \|\mathcal{B}(\boldsymbol{u}(t^{n+1}))\|^2_2\right\},  \\
    			\tilde{C}_{\phi} &:= C \max\left\{\frac{36L^2}{\epsilon_{\phi}}, \; \frac{12L^2}{\epsilon_{\phi}} + 2\|\nabla \mathcal{B}(\boldsymbol{u}^{n+1})\|^2 + \|\nabla \mathcal{B}(\mu^{n+1})\|^2\right\}.
    		\end{aligned}
    	\end{equation}
    	Using the definition of $\mathcal{B}(s^{n+1})$ to obtain, for some $s$, 
    	\begin{equation}
    		\begin{aligned}
    			\|\mathcal{B}(s^{n+1})\|^2 
    			&= \left\|\frac32s^{n+1} - s^{n} + \frac12s^{n-1}\right\|^2  \\
    			&= \frac14 \left\|s^{n+1} + 2s^{n+1} - s^{n} - (s^{n} - s^{n-1})\right\|^2 \\
    			&\leq \frac34 \left(\left\|s^{n+1}\right\|^2 + \left\|2s^{n+1} - s^{n}\right\|^2 + \left\|s^{n} - s^{n-1}\right\|^2\right) \\
    			& \leq \frac34 \left(\left\|s^{n+1}\right\|^2 + \left\|2s^{n+1} - s^{n}\right\|^2 + \left\|s^{n+1} - s^{n}\right\|^2\right) \\
    			& \quad + \frac34 \left(\left\|s^{n}\right\|^2 + \left\|2s^{n} - s^{n-1}\right\|^2 + \left\|s^{n} - s^{n-1}\right\|^2\right),
    		\end{aligned}
    	\end{equation}
        and then, using Lemma \ref{Lemma_Estimation_Truncation_error}, setting $\epsilon_{\phi} = 1/4$, multiplying $4\Delta t$ on both sides and adding from $n=1$ to $n=m(\leq N-1)$ in \eqref{lemma_main_results}, the above inequality becomes
        \begin{equation}
        	\begin{aligned}
        		&\|\nabla e_{\phi}^{m+1}\|^2+\|\nabla \left(2e_{\phi}^{m+1}-e_{\phi}^m\right)\|^2
        		+\|\nabla \left(e_{\phi}^{m+1}-e_{\phi}^m\right)\|^2  \\ 
        		& \qquad  +\|e_{\boldsymbol{u}}^{m+1}\|^2+\|2e_{\boldsymbol{u}}^{m+1}-e_{\boldsymbol{u}}^m\|^2
        		+\|e_{\boldsymbol{u}}^{m+1}-e_{\boldsymbol{u}}^m\|^2 \\
        		& \qquad + 3\Delta t\sum_{n=1}^{m}\left(\|\nabla \left(e_{\phi}^{n+1}-2e_{\phi}^n+e_{\phi}^{n-1}\right)\|^2
        		+\| e_{\boldsymbol{u}}^{n+1}-2e_{\boldsymbol{u}}^n+e_{\boldsymbol{u}}^{n-1}\|^2\right) \\
        		& \qquad + \Delta t\sum_{n=1}^{m}\|\nabla\mathcal{B}(e_{\mu}^{n+1})\|^2 + 2\Delta t\sum_{n=1}^{m}\|\nabla\mathcal{B}(e_{\boldsymbol{u}}^{n+1})\|^2 + 4S\Delta t\sum_{n=1}^{m}\|\mathcal{A}(e_{\phi}^{n+1})\|^2  \\
        		& \leq \|\nabla e_{\phi}^{1}\|^2+\|\nabla \left(2e_{\phi}^{1}-e_{\phi}^0\right)\|^2
        		+ \|\nabla \left(e_{\phi}^{1}-e_{\phi}^0\right)\|^2   + \|e_{\boldsymbol{u}}^{1}\|^2+\|2e_{\boldsymbol{u}}^{1}-e_{\boldsymbol{u}}^0\|^2 + \|e_{\boldsymbol{u}}^{1}-e_{\boldsymbol{u}}^0\|^2  \\
        		& \qquad + C \Delta t\sum_{n=1}^{m} \hat{C}_{\Delta t} + C_{2} \Delta t\sum_{n=1}^{m} (\Delta t)^4 + \Delta t\sum_{n=0}^{m} C^* \bigg(\|e_{\boldsymbol{u}}^{n+1}\|^2+\|2e_{\boldsymbol{u}}^{n+1}-e_{\boldsymbol{u}}^n\|^2 +\|e_{\boldsymbol{u}}^{n+1}-e_{\boldsymbol{u}}^n\|^2  \\ 
        		& \qquad  + \|\nabla e_{\phi}^{n+1}\|^2+\|\nabla \left(2e_{\phi}^{n+1}-e_{\phi}^n\right)\|^2 + \|\nabla \left(e_{\phi}^{n+1}-e_{\phi}^n\right)\|^2\bigg),
        	\end{aligned}
        \end{equation}
        where 
        \begin{equation}
        	\begin{aligned}
        		\hat{C}_{\Delta t} :=& C  \left\|\frac{1}{\Delta t}\mathcal{A}\left(\phi(t^{n+1})\right)- \phi_t(t^{n+1})\right\| + C \|\Delta \left(\mathcal{B}\left(\mu\left(t^{n+1}\right) \right)- \mu(t^{n+1})\right)\|   \\
        		& \quad + C \|\mathcal{B}\left(\boldsymbol{u}(t^{n+1})\right) - \boldsymbol{u}(t^{n+1})\|_1 \cdot \|\mathcal{B}\left(\phi(t^{n+1})\right)\|_2  \\
        		& \quad + C \|\boldsymbol{u}(t^{n+1})\|_1 \cdot \|\phi(t^{n+1}) - \mathcal{B}\left(\phi(t^{n+1})\right)\|_2  \\
        		& \quad + C \|\nabla \left(\mathcal{B}(\mu(t^{n+1}))-\mu(t^{n+1})\right)\| + C \|\nabla \Delta \left(\mathcal{B}(\phi(t^{n+1}))-\phi(t^{n+1})\right)\| \\
        		& \quad + C\left\|\frac{1}{\Delta t}\mathcal{A}(\boldsymbol{u}(t^{n+1}))-\boldsymbol{u}_t(t^{n+1})\right\| + C\left\|\Delta\left(\mathcal{B}(\boldsymbol{u}(t^{n+1}))-\boldsymbol{u}(t^{n+1})\right)\right\|  \\
        		& \quad + C\|\mathcal{B}\left(\boldsymbol{u}(t^{n+1})\right)\|_1 \cdot \left\|\mathcal{B}\left(\boldsymbol{u}(t^{n+1})\right) - \boldsymbol{u}(t^{n+1})\right\|_2  \\
        		& \quad + C\|\boldsymbol{u}(t^{n+1})\|_2 \cdot \left\|\mathcal{B}\left(\boldsymbol{u}(t^{n+1})\right) - \boldsymbol{u}(t^{n+1})\right\|_1  \\
        		& \quad + C\|\mathcal{B}\left(\mu(t^{n+1})\right)\|_1 \cdot \left\|\mathcal{B}\left(\phi(t^{n+1})\right) - \phi(t^{n+1}) \right\|_2  \\
        		& \quad + C\|\phi(t^{n+1})\|_2 \cdot \left\|\mathcal{B}\left(\mu(t^{n+1})\right) - \mu(t^{n+1})\right\|_1  \\
        		& \quad + C \|\nabla\mathcal{B}(p(t^{n+1}))-p(t^{n+1}))\|,  \\
        		C_{2} :=& C_{1} + CL^2\left(\Delta t\right)^4 \|\phi_{tt}\|^2_{L^{\infty}(0,T;H^1)},  \\
        		C^* :=& \max\left\{\frac{10}{4}\tilde{C}_u, \tilde{C}_{\phi}\right\}.
        	\end{aligned}
        \end{equation}
    	Finally, using Lemma \ref{lemma_consistency_error}, regularity assumptions \ref{eq_regularity_001}, and stability in \ref{lemma_boundness_expression}, and assuming $\Delta t$ small enough such that $C^*\Delta t \leq 1$, where $C^*$ is defined above, and then applying the discrete Gr\"{o}nwall lemma \ref{lemma_gronwall_inequalities}, we can obtain the results of Theorem \ref{theorem_main_results}.
    \end{proof}

	\section{Numerical result}\label{section_numerical_tests}
	 In this section, we verify the convergence and stability of the fully implicit scheme \eqref{eq_FIBETF_scheme0001}-\eqref{eq_FIBETF_scheme0002}, and semi-implicit scheme \eqref{eq_algorithm3_1}-\eqref{eq_algorithm3_2} through several numerical experiments, mainly focusing on the results of Theorem \ref{theorem_unconditional_stability} and Theorem \ref{theorem_main_results} . We use the finite elements methods to discretize spatial variables, where the $P_1$ elements are used for $\phi^{n+1}$ and $\mu^{n+1}$, and the $\boldsymbol{P}_2\times P_1$ elements for $\boldsymbol{u}^{n+1}$ and $p^{n+1}$. We set $L = 2$ in \eqref{eq_phi_condition}, and the stabilization parameter $S = 3L/\Delta t = 3/\Delta t$ in the second-order linear TF scheme \eqref{eq_algorithm3_1}-\eqref{eq_algorithm3_2}; in particular, in the convergence rate verification test, $S$ corresponds fixedly to the smallest $\Delta t$. The simulations of the last two practical examples are obtained via scheme \eqref{eq_algorithm3_1}-\eqref{eq_algorithm3_2}.
	 
	 \subsection{Convergence rate verification}
	 The manufactured solutions are chosen from \cite{2022_ChenYaoyao_CHNS_2022_AMC} and takes the form
	 $$
	 \left\{
	 \begin{aligned}
	 	\phi(t,x,y) &= 2 + \sin(t)\cos(\pi x)\cos(\pi y),  \\
	 	\boldsymbol{u}(t,x,y) &= \left[\pi\sin(\pi x)^2\sin(2\pi y)\sin(t), -\pi\sin(\pi y)^2\sin(2\pi x)\sin(t)\right]^{\text{T}},  \\
	 	p(t,x,y) &= \cos(\pi x)\sin(\pi y)\sin(t),
	 \end{aligned}\right.
	 $$
	 and the exact chemical potential $\mu(t,x,y)$ is obtained by its definition, i.e., 
	 $$
	 \begin{aligned}
	 	\mu(t,x,y) :=&  -\epsilon \Delta \phi(t,x,y) + \frac{1}{\epsilon} \left(\phi(t,x,y)^3 - \phi(t,x,y)\right) \\
	 	=&~ 2\epsilon \pi^2 \cos(\pi x)\cos(\pi y) \sin(t)  \\
	 	&-\frac{1}{\epsilon}\left[\cos(\pi x)\cos(\pi y) \sin(t) - \big(\cos(\pi x)\cos(\pi y) \sin(t) + 2\big)^3 + 2 \right].
	 \end{aligned}
	 $$
	 The domain is defined as $\Omega = [0,1]\times[0,1]$, and the time is $T = 1$. The physical parameters are selected as $M = 0.01$, $\gamma = 0.01$, $\epsilon = 0.2$, $\nu = 1$. We set $S = 1$ in the first-order semi-implicit scheme \eqref{eq_BE_scheme0002}, and the time-step $\Delta t$ is set up as $\Delta t = h$ to balance the convergence rates between time and space. This experiment aims to validate the second-order temporal error convergence of the backward Euler time filtering (BE-TF) algorithm. To highlight the effects of the TF technique in enhancing temporal convergence rates, we first show the first-order convergence of for the nonlinear \textbf{Algorithm 1} in \autoref{FullyImplicitBE_error_convRates} and the linear \textbf{Algorithm 2} in \autoref{BE_error_convRates}. Then, we present the errors and convergence rates of the nonlinear and linear BE-TF schemes, i.e., \textbf{Algorithm 3} and \textbf{Algorithm 4}, with the pressure filtered (Option A) in \autoref{BETF_FullyImplicit_2_error_convRates}, \autoref{BETF_LinearlyImplicit_2_error_convRates} and without the pressure filtered (Option B) in \autoref{BETF_FullyImplicit_2_error_convRates_withoutPresTimeFilter}, \autoref{BETF_LinearlyImplicit_2_error_convRates_withoutPresTimeFilter}, respectively. Moreover, the numerical test for the variable stepsize \textbf{Algorithm 5} is presented in \autoref{ConvRate_variable_stepsize_timeFilter}. By those comparisons, we confirm that the TF technique indeed elevates the temporal accuracy of the BE method from first- to second-order, for the constant and variable stepsize. Moreover, the numerical verification of adaptive stepsize \textbf{Algorithm 6} is presented in \autoref{ConvRate_adaptive_stepsize_timeFilter}, which confirms the effectiveness of TF in constructing adaptive stepsize algorithms.
	 
     \begin{table}[h!]
     	\centering
     	\fontsize{10}{10}
     	\begin{threeparttable}
     		\caption{Errors and convergence orders of the FIMBE scheme.}
     		\label{FullyImplicitBE_error_convRates}
     		\begin{tabular}{c|c|c|c|c|c|c|c|c}
     			\toprule
     			\multirow{2.5}{*}{$\Delta t$}  & \multicolumn{2}{c}{$\left\|\phi^N - \phi^N_h\right\|_{L^2}$} & \multicolumn{2}{c}{$\left\|\mu^N - \mu^N_h\right\|_{L^2}$} & \multicolumn{2}{c}{$\left\|\boldsymbol{u}^N - \boldsymbol{u}^N_h\right\|_{L^2}$} & \multicolumn{2}{c}{$\left\|p^N - p^N_h\right\|_{L^2}$} \cr
     			\cmidrule(lr){2-3} \cmidrule(lr){4-5} \cmidrule(lr){6-7}  \cmidrule(lr){8-9}
     			& error & rate & error & rate & error & rate & error & rate \cr
     			\midrule
     			1/4    & 4.1463e$-$02 &  -       & 2.9460e$-$00  & -       & 4.4755e$-$02  & -      & 3.0406e$-$01 & -      \cr
     			1/8    & 1.3583e$-$02 &  1.6101  & 9.9744e$-$01  & 1.5625  & 5.8930e$-$03  & 2.9250 & 2.5481e$-$02 & 3.5768 \cr
     			1/16   & 4.9467e$-$03 &  1.4572  & 3.4718e$-$01  & 1.5226  & 1.0236e$-$03  & 2.5254 & 4.4821e$-$03 & 2.5072 \cr
     			1/32   & 1.9784e$-$03 &  1.3221  & 1.3092e$-$01  & 1.4070  & 3.5189e$-$04  & 1.5404 & 1.8951e$-$03 & 1.2419 \cr
     			1/64   & 8.6787e$-$04 &  1.1888  & 5.4757e$-$02  & 1.2576  & 1.6846e$-$04  & 1.0627 & 8.9368e$-$04 & 1.0844 \cr
     			1/128  & 4.0468e$-$04 &  1.1007  & 2.4769e$-$02  & 1.1445  & 8.4019e$-$05  & 1.0036 & 4.3346e$-$04 & 1.0439 \cr
     			\bottomrule
     		\end{tabular}
     	\end{threeparttable}
     \end{table}
     \begin{table}[h!]
     	\centering
     	\fontsize{10}{10}
     	\begin{threeparttable}
     		\caption{Errors and convergence orders of the SIMBE scheme.}
     		\label{BE_error_convRates}
     		\begin{tabular}{c|c|c|c|c|c|c|c|c}
     			\toprule
     			\multirow{2.5}{*}{$\Delta t$}  & \multicolumn{2}{c}{$\left\|\phi^N - \phi^N_h\right\|_{L^2}$} & \multicolumn{2}{c}{$\left\|\mu^N - \mu^N_h\right\|_{L^2}$} & \multicolumn{2}{c}{$\left\|\boldsymbol{u}^N - \boldsymbol{u}^N_h\right\|_{L^2}$} & \multicolumn{2}{c}{$\left\|p^N - p^N_h\right\|_{L^2}$} \cr
     			\cmidrule(lr){2-3} \cmidrule(lr){4-5} \cmidrule(lr){6-7}  \cmidrule(lr){8-9}
     			& error & rate & error & rate & error & rate & error & rate \cr
     			\midrule
     			1/4    & 5.4034e$-$02 &  -       & 2.3858e$-$00  & -       & 4.4746e$-$02  & -      & 5.4907e$-$01 & -      \cr
     			1/8    & 2.7351e$-$02 &  0.9823  & 5.5371e$-$01  & 2.1072  & 5.9016e$-$03  & 2.9224 & 1.9337e$-$01 & 1.5057 \cr
     			1/16   & 1.5472e$-$02 &  0.8219  & 1.6034e$-$01  & 1.7880  & 1.0383e$-$03  & 2.5069 & 8.9099e$-$02 & 1.1179 \cr
     			1/32   & 8.2934e$-$03 &  0.8996  & 7.8600e$-$02  & 1.0286  & 3.6209e$-$04  & 1.5198 & 4.3259e$-$02 & 1.0424 \cr
     			1/64   & 4.2910e$-$03 &  0.9506  & 4.4297e$-$02  & 0.8273  & 1.7364e$-$04  & 1.0602 & 2.1342e$-$02 & 1.0193 \cr
     			1/128  & 2.1818e$-$03 &  0.9758  & 2.3905e$-$02  & 0.8899  & 8.6582e$-$05  & 1.0040 & 1.0601e$-$02 & 1.0095 \cr
     			\bottomrule
     		\end{tabular}
     	\end{threeparttable}
     \end{table}
     \begin{table}[h!]
     	\centering
     	\fontsize{10}{10}
     	\begin{threeparttable}
     		\caption{Errors and convergence orders of the nonlinear FIMBE-TF algorithm with the pressure filtered.}\label{BETF_FullyImplicit_2_error_convRates}
     		\begin{tabular}{c|c|c|c|c|c|c|c|c}
     			\toprule
     			\multirow{2.5}{*}{$\Delta t$}  & \multicolumn{2}{c}{$\left\|\phi^N - \phi^N_h\right\|_{L^2}$} & \multicolumn{2}{c}{$\left\|\mu^N - \mu^N_h\right\|_{L^2}$} & \multicolumn{2}{c}{$\left\|\boldsymbol{u}^N - \boldsymbol{u}^N_h\right\|_{L^2}$} & \multicolumn{2}{c}{$\left\|p^N - p^N_h\right\|_{L^2}$} \cr
     			\cmidrule(lr){2-3} \cmidrule(lr){4-5} \cmidrule(lr){6-7}  \cmidrule(lr){8-9}
     			& error & rate & error & rate & error & rate & error & rate \cr
     			\midrule
     			1/4    & 3.6396e$-$02 &  -       & 2.6110e$-$00  & -       & 3.5544e$-$02  & -      & 3.2404e$-$01 & -      \cr
     			1/8    & 8.6714e$-$03 &  2.0695  & 6.7843e$-$01  & 1.9443  & 7.8231e$-$03  & 2.1838 & 2.5500e$-$02 & 3.6676 \cr
     			1/16   & 2.1569e$-$03 &  2.0073  & 1.7584e$-$01  & 1.9480  & 2.0723e$-$03  & 1.9165 & 2.6372e$-$03 & 3.2735 \cr
     			1/32   & 5.3649e$-$04 &  2.0073  & 4.4181e$-$02  & 1.9927  & 5.3394e$-$04  & 1.9565 & 4.9211e$-$04 & 2.4220 \cr
     			1/64   & 1.3357e$-$04 &  2.0060  & 1.1033e$-$02  & 2.0015  & 1.3526e$-$04  & 1.9810 & 1.1879e$-$04 & 2.0506 \cr
     			1/128  & 3.3308e$-$05 &  2.0037  & 2.7541e$-$03  & 2.0022  & 3.4017e$-$05  & 1.9914 & 2.9642e$-$05 & 2.0027 \cr
     			\bottomrule
     		\end{tabular}
     	\end{threeparttable}
     \end{table}
     \begin{table}[h!]
     	\centering
     	\fontsize{10}{10}
     	\begin{threeparttable}
     		\caption{Errors and convergence orders of the nonlinear FIMBE-TF algorithm without the pressure filtered}\label{BETF_FullyImplicit_2_error_convRates_withoutPresTimeFilter}
     		\begin{tabular}{c|c|c|c|c|c|c|c|c}
     			\toprule
     			\multirow{2.5}{*}{$\Delta t$}  & \multicolumn{2}{c}{$\left\|\phi^N - \phi^N_h\right\|_{L^2}$} & \multicolumn{2}{c}{$\left\|\mu^N - \mu^N_h\right\|_{L^2}$} & \multicolumn{2}{c}{$\left\|\boldsymbol{u}^N - \boldsymbol{u}^N_h\right\|_{L^2}$} & \multicolumn{2}{c}{$\left\|p^N - p^N_h\right\|_{L^2}$} \cr
     			\cmidrule(lr){2-3} \cmidrule(lr){4-5} \cmidrule(lr){6-7}  \cmidrule(lr){8-9}
     			& error & rate & error & rate & error & rate & error & rate \cr
     			\midrule
     			1/4    & 3.6396e$-$02 &  -       & 2.6110e$-$00  & -       & 3.5544e$-$02  & -      & 3.0602e$-$01 & -      \cr
     			1/8    & 8.6714e$-$03 &  2.0695  & 6.7843e$-$01  & 1.9443  & 7.8231e$-$03  & 2.1838 & 2.5207e$-$02 & 3.6017 \cr
     			1/16   & 2.1569e$-$03 &  2.0073  & 1.7584e$-$01  & 1.9480  & 2.0723e$-$03  & 1.9165 & 2.5113e$-$03 & 3.3273 \cr
     			1/32   & 5.3649e$-$04 &  2.0073  & 4.4181e$-$02  & 1.9927  & 5.3394e$-$04  & 1.9565 & 4.4881e$-$04 & 2.4843 \cr
     			1/64   & 1.3357e$-$04 &  2.0060  & 1.1033e$-$02  & 2.0015  & 1.3526e$-$04  & 1.9810 & 1.0734e$-$04 & 2.0639 \cr
     			1/128  & 3.3308e$-$05 &  2.0037  & 2.7541e$-$03  & 2.0022  & 3.4017e$-$05  & 1.9914 & 2.6746e$-$04 & 2.0048 \cr
     			\bottomrule
     		\end{tabular}
     	\end{threeparttable}
     \end{table}
	 \begin{table}[h!]
	 	\centering
	 	\fontsize{10}{10}
	 	\begin{threeparttable}
	 		\caption{Errors and convergence orders of the linear SIMBE-TF algorithm with the pressure filtered.}\label{BETF_LinearlyImplicit_2_error_convRates}
	 		\begin{tabular}{c|c|c|c|c|c|c|c|c}
	 			\toprule
	 			\multirow{2.5}{*}{$\Delta t$}  & \multicolumn{2}{c}{$\left\|\phi^N - \phi^N_h\right\|_{L^2}$} & \multicolumn{2}{c}{$\left\|\mu^N - \mu^N_h\right\|_{L^2}$} & \multicolumn{2}{c}{$\left\|\boldsymbol{u}^N - \boldsymbol{u}^N_h\right\|_{L^2}$} & \multicolumn{2}{c}{$\left\|p^N - p^N_h\right\|_{L^2}$} \cr
	 			\cmidrule(lr){2-3} \cmidrule(lr){4-5} \cmidrule(lr){6-7}  \cmidrule(lr){8-9}
	 			& error & rate & error & rate & error & rate & error & rate \cr
	 			\midrule
	 			1/4    & 2.2447e$-$01 &  -       & 11.471e$-$00  & -       & 3.5603e$-$02  & -      & 3.0612e$-$01 & -      \cr
	 			1/8    & 1.7215e$-$01 &  0.3829  & 10.424e$-$00  & 0.1381  & 8.1150e$-$03  & 2.1333 & 6.9073e$-$02 & 2.1479 \cr
	 			1/16   & 6.2742e$-$02 &  1.4562  & 3.3290e$-$00  & 1.6467  & 2.0473e$-$03  & 1.9869 & 2.2671e$-$02 & 1.6072 \cr
	 			1/32   & 1.6435e$-$02 &  1.9326  & 7.3751e$-$01  & 2.1744  & 5.3035e$-$04  & 1.9487 & 5.6622e$-$03 & 2.0014 \cr
	 			1/64   & 4.1262e$-$03 &  1.9939  & 1.7934e$-$01  & 2.0400  & 1.3580e$-$04  & 1.9655 & 1.4217e$-$03 & 1.9937 \cr
	 			1/128  & 1.0322e$-$03 &  1.9991  & 4.4598e$-$02  & 2.0076  & 3.4257e$-$05  & 1.9870 & 3.5639e$-$04 & 1.9961 \cr
	 			\bottomrule
	 		\end{tabular}
	 	\end{threeparttable}
	 \end{table}
	 \begin{table}[h!]
	 	\centering
	 	\fontsize{10}{10}
	 	\begin{threeparttable}
	 		\caption{Errors and convergence orders of the linear SIMBE-TF algorithm without the pressure filtered}\label{BETF_LinearlyImplicit_2_error_convRates_withoutPresTimeFilter}
	 		\begin{tabular}{c|c|c|c|c|c|c|c|c}
	 			\toprule
	 			\multirow{2.5}{*}{$\Delta t$}  & \multicolumn{2}{c}{$\left\|\phi^N - \phi^N_h\right\|_{L^2}$} & \multicolumn{2}{c}{$\left\|\mu^N - \mu^N_h\right\|_{L^2}$} & \multicolumn{2}{c}{$\left\|\boldsymbol{u}^N - \boldsymbol{u}^N_h\right\|_{L^2}$} & \multicolumn{2}{c}{$\left\|p^N - p^N_h\right\|_{L^2}$} \cr
	 			\cmidrule(lr){2-3} \cmidrule(lr){4-5} \cmidrule(lr){6-7}  \cmidrule(lr){8-9}
	 			& error & rate & error & rate & error & rate & error & rate \cr
	 			\midrule
	 			1/4    & 2.2447e$-$01 &  -       & 11.471e$-$00  & -       & 3.5603e$-$02  & -      & 2.8731e$-$01 & -      \cr
	 			1/8    & 1.7215e$-$01 &  0.3829  & 10.424e$-$00  & 0.1381  & 8.1150e$-$03  & 2.1333 & 6.7899e$-$02 & 2.0812 \cr
	 			1/16   & 6.2742e$-$02 &  1.4562  & 3.3290e$-$00  & 1.6467  & 2.0473e$-$03  & 1.9869 & 2.2477e$-$02 & 1.5950 \cr
	 			1/32   & 1.6435e$-$02 &  1.9326  & 7.3751e$-$01  & 2.1744  & 5.3035e$-$04  & 1.9487 & 5.6529e$-$03 & 1.9914 \cr
	 			1/64   & 4.1262e$-$03 &  1.9939  & 1.7934e$-$01  & 2.0400  & 1.3580e$-$04  & 1.9655 & 1.4206e$-$03 & 1.9925 \cr
	 			1/128  & 1.0322e$-$03 &  1.9991  & 4.4598e$-$02  & 2.0076  & 3.4257e$-$05  & 1.9870 & 3.5615e$-$04 & 1.9959 \cr
	 			\bottomrule
	 		\end{tabular}
	 	\end{threeparttable}
	 \end{table}
     \begin{table}[h!]
     	\centering
     	\fontsize{10}{10}
     	\begin{threeparttable}
     		\caption{Errors and convergence orders of the VSBE-TF scheme on the graded mesh with time $t_n = \left(n/N\right)^2$ for different $N$.}\label{ConvRate_variable_stepsize_timeFilter}
     		\begin{tabular}{c|c|c|c|c|c|c|c|c}
     			\toprule
     			\multirow{2.5}{*}{$N$}  & \multicolumn{2}{c}{$\max\limits_{n}\left\|\phi^n - \phi^n_h\right\|_{L^2}$} & \multicolumn{2}{c}{$\max\limits_{n}\left\|\mu^n - \mu^n_h\right\|_{L^2}$} & \multicolumn{2}{c}{$\max\limits_{n}\left\|\boldsymbol{u}^n - \boldsymbol{u}^n_h\right\|_{L^2}$} & \multicolumn{2}{c}{$\max\limits_{n}\left\|p^n - p^n_h\right\|_{L^2}$} \cr
     			\cmidrule(lr){2-3} \cmidrule(lr){4-5} \cmidrule(lr){6-7}  \cmidrule(lr){8-9}
     			& error & rate & error & rate & error & rate & error & rate \cr
     			\midrule
     			8    & 8.0484e$-$03 &  -       & 6.2034e$-$01  & -       & 2.3450e$-$02  & -      & 1.2258e$-$02 & -      \cr
     			16   & 1.7408e$-$03 &  2.2090  & 9.7741e$-$02  & 2.6660  & 7.2759e$-$03  & 1.6884 & 3.3022e$-$03 & 1.8923 \cr
     			32   & 4.7870e$-$04 &  1.8625  & 2.8057e$-$02  & 1.8006  & 1.9870e$-$03  & 1.8724 & 8.8731e$-$04 & 1.8959 \cr
     			64   & 1.1449e$-$04 &  2.0639  & 6.8313e$-$03  & 2.0381  & 5.1699e$-$04  & 1.9424 & 2.3470e$-$04 & 1.9186 \cr
     			\bottomrule
     		\end{tabular}
     	\end{threeparttable}
     \end{table}
     \begin{table}[h!]
     	\centering
     	\fontsize{10}{10}
     	\begin{threeparttable}
     		\caption{Errors and convergence orders of the ASBE-TF scheme.}\label{ConvRate_adaptive_stepsize_timeFilter}
     		\begin{tabular}{c|c|c|c|c|c|c|c|c}
     			\toprule
     			\multirow{2.5}{*}{$\overline{\Delta t}$}  & \multicolumn{2}{c}{$\left\|\phi^N - \phi^N_h\right\|_{L^2}$} & \multicolumn{2}{c}{$\left\|\mu^N - \mu^N_h\right\|_{L^2}$} & \multicolumn{2}{c}{$\left\|\boldsymbol{u}^N - \boldsymbol{u}^N_h\right\|_{L^2}$} & \multicolumn{2}{c}{$\left\|p^N - p^N_h\right\|_{L^2}$} \cr
     			\cmidrule(lr){2-3} \cmidrule(lr){4-5} \cmidrule(lr){6-7}  \cmidrule(lr){8-9}
     			& error & rate & error & rate & error & rate & error & rate \cr
     			\midrule
     			1/16    & 2.0466e$-$01 &  -       & 12.361e$-$00  & -       & 8.2843e$-$02  & -      & 6.3417e$-$02 & -      \cr
     			1/32    & 1.4331e$-$01 &  0.5141  & 8.2971e$-$00  & 0.5751  & 3.1732e$-$02  & 1.3844 & 4.9143e$-$02 & 0.3679 \cr
     			1/77    & 2.8214e$-$02 &  1.8509  & 1.0651e$-$00  & 2.3379  & 9.7319e$-$03  & 1.3460 & 8.4468e$-$03 & 2.0055 \cr
     			1/167   & 3.9732e$-$03 &  2.5320  & 1.5149e$-$01  & 2.5191  & 3.0392e$-$03  & 1.5033 & 1.4308e$-$03 & 2.2935 \cr
     			1/505   & 7.7903e$-$04 &  1.4724  & 4.2378e$-$02  & 1.1512  & 9.6056e$-$04  & 1.0409 & 4.0012e$-$04 & 1.1515 \cr
     			\bottomrule
     		\end{tabular}
     	\end{threeparttable}
     \end{table}
	\subsection{Shape relaxation}
	In this numerical experiment, we set the domain $\Omega=[0,1]\times[0,1]$, mesh size $h=1/64$, $\Delta t = 0.1$ and take the rotational boundary condition $\boldsymbol{u}=(y-0.5,-x+0.5)$ on $\partial\Omega$. In this example, we choose $\epsilon=0.02$, $M=0.01$, $\gamma=0.01$, $\nu=1$, and the performance of the scheme \eqref{eq_algorithm3_1}-\eqref{eq_algorithm3_2} with the critical phase field initial data which is the $\phi^0=1$ in a polygonal subdomain with re-entrant corners and $\phi^0=-1$ in the remaining part of $\Omega$.
	The initial velocity are given by $\boldsymbol{u}^0=\left(y-0.5,-x+0.5\right)$. This problem has been studied numerically in \cite{2008_Kay_David_and_Welford_Richard_Finite_element_approximation_of_a_Cahn_Hilliard_Navier_Stokes_system}, and we run this example up to final time $T=10$, and record the several snapshots of phase function in 
	\autoref{PhaseSnap_RotaBounCond_dt01_20250617.eps}. From this figure, one can see that the cross-shaped area gradually degenerates into a circle under the action of the rotating boundary condition.
%
	\begin{figure}[h!]
		\centering
		\includegraphics[width=1\linewidth]{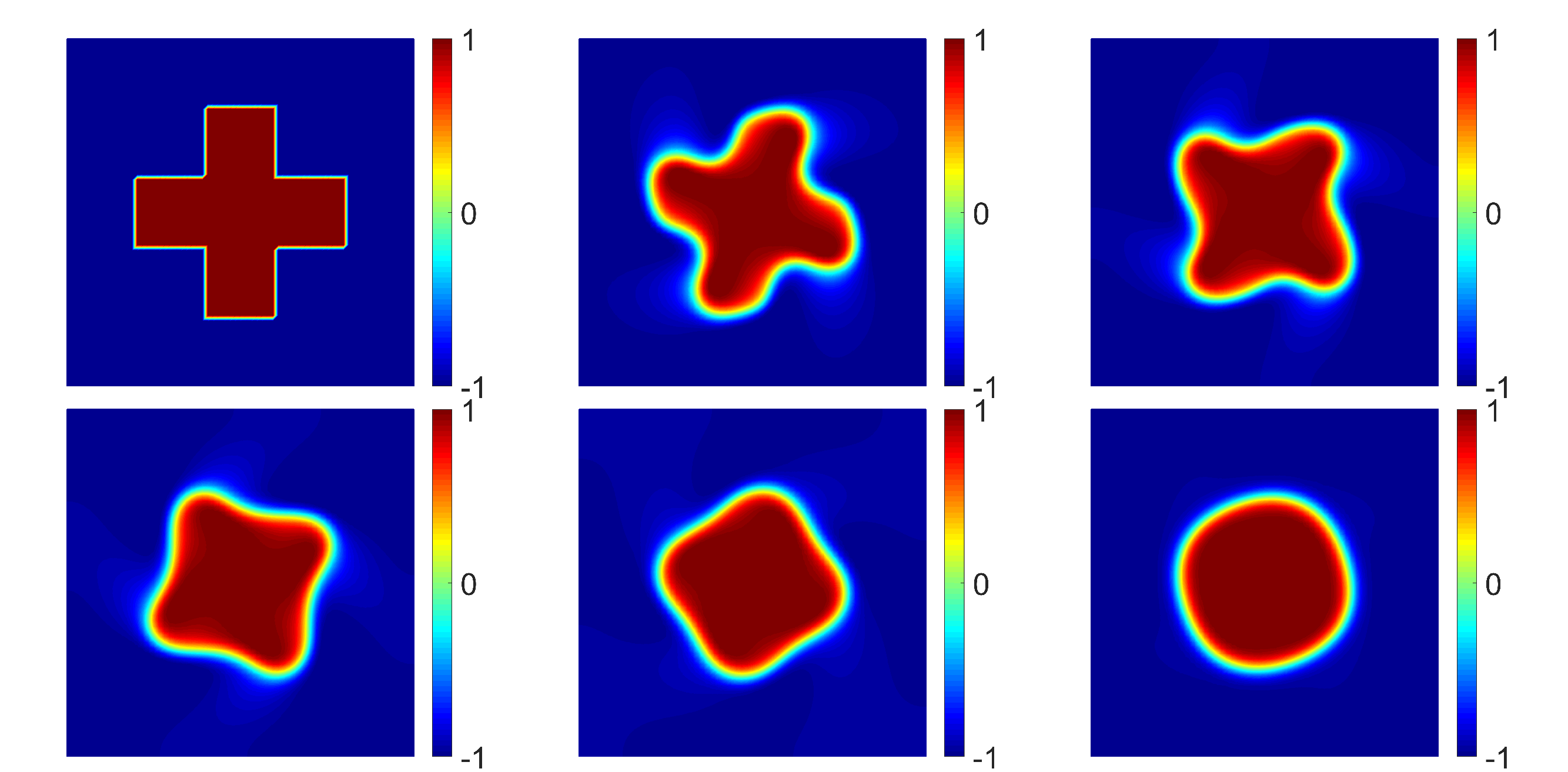}
		\caption{Snapshots of phase function at difference times from left to right row by row with $t = [0,~ 1,~ 2,~ 3,~ 5,~ 10]$, respectively.}
		\label{PhaseSnap_RotaBounCond_dt01_20250617.eps}
	\end{figure}
	\subsection{Shrinking circular bubble}
	In this example, we consider the simulation of the shrinking process of a circular bubble, as a further confirmation of the ability of our newly TF scheme to simulate the practical problems stably and accurately, and this test have also been conducted in \cite{2024_CHEN_chuanjun_SCM}. The computational domain is $\Omega=[0,2\pi]\times[0,2\pi]$, and mesh size $h=2\pi/64$, time-step $\Delta t = 0.1$. We simulate this process until the final time $T = 15$. The physical parameters are chosen as:
	$\epsilon=0.15$, $M=0.4$, $\gamma=0.01$, and $\nu=1$.
	The initial values are set as
	$$
	\begin{aligned}
		\phi_0 &= 1 + \tanh\left(\frac{1.4 - \sqrt{(x+0.8-\pi)^2+(y-\pi)^2}}{1.5\epsilon}\right) + \tanh\left(0.5 - \frac{\sqrt{(x-1.7-\pi)^2+(y-\pi)^2}}{1.5\epsilon}\right).  \\
		\boldsymbol{u}^0 &= \boldsymbol{0}, \quad p^0 = 0.
	\end{aligned}
	$$
	From the \autoref{Figure_Stabf_TwoBubbMerg_20250701}, we can see that the two bubbles gradually converge towards a steady circular bubble, which validates the ability of TF scheme.
	\begin{figure}[h!]
		\centering
		\includegraphics[width=1\linewidth]{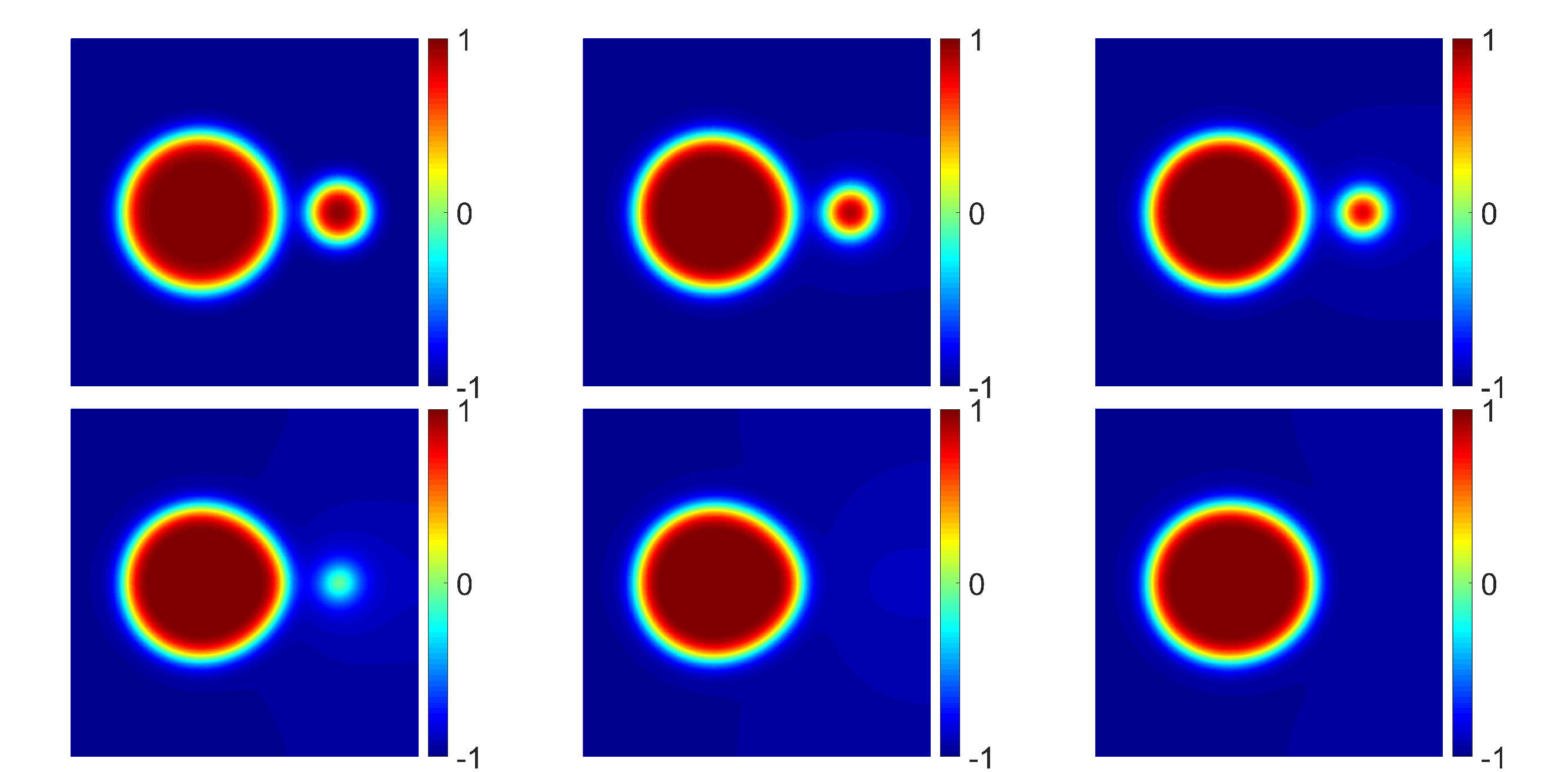}
		\caption{Snapshots of phase function at difference times from left to right row by row with $t = [0, ~2,~ 3,~ 5, ~7, ~15]$, respectively.}
		\label{Figure_Stabf_TwoBubbMerg_20250701}
	\end{figure}
    \subsection{Unconditional stability}
    To demonstrate the unconditional stability of the newly TF scheme, we run the two above practical experiments again with the same settings, except the variable time-step $\Delta t$. From the \autoref{Figure_UncondStab_20250611}, one can see that the TF technique indeed maintain the unconditional energy-stable of first-order BE scheme, just as analyzed in Theorem \autoref{theorem_unconditional_stability}.
    \begin{figure}[h!]
    	\centering
    	\includegraphics[width=1\linewidth]{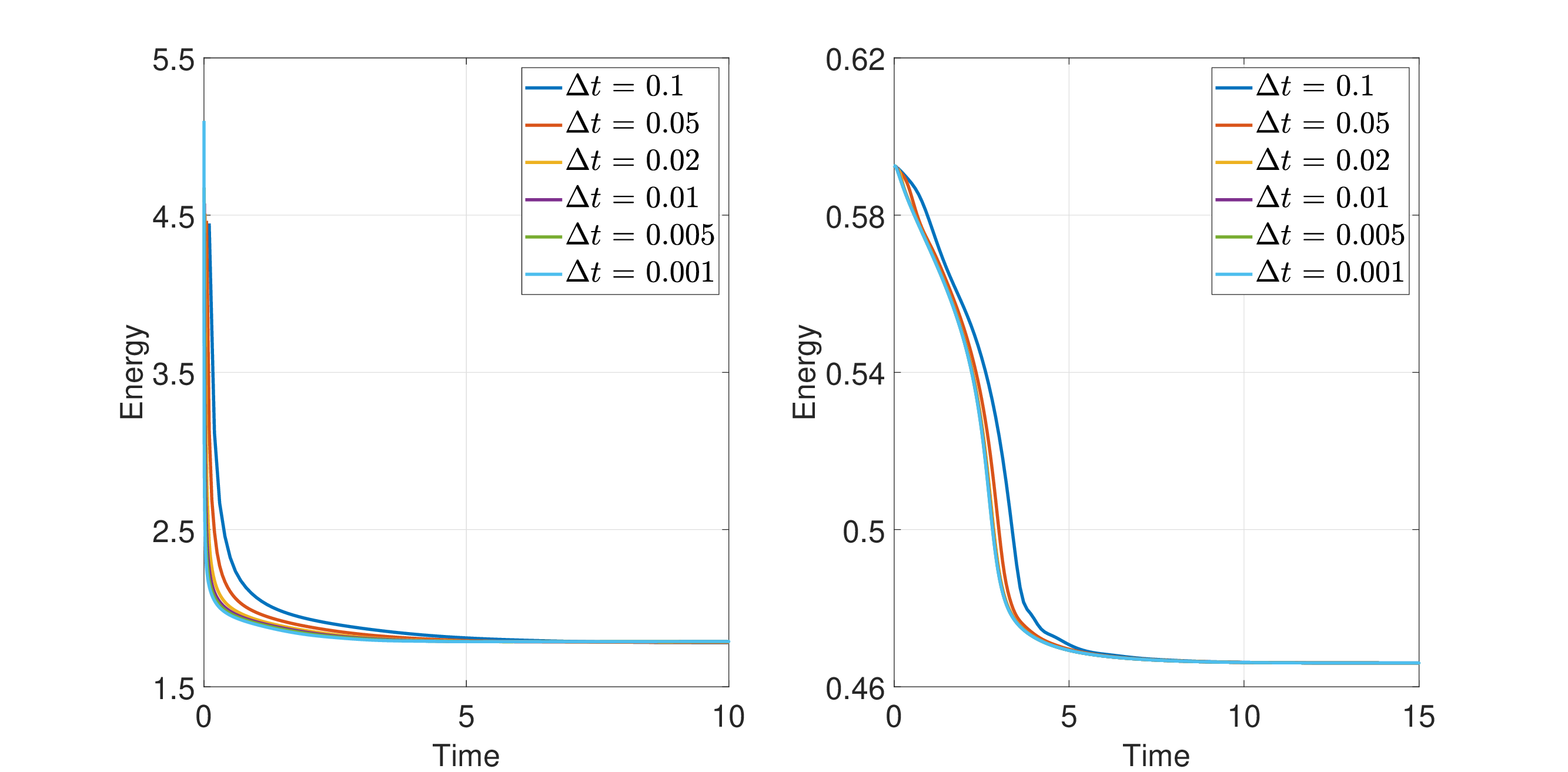}
    	\caption{Evolution of energy for shape relaxation experiment (left) and shrinking circular bubble experiment (right).}
    	\label{Figure_UncondStab_20250611}
    \end{figure}
	\section{Conclusion}\label{section-5}
	In this article, we have presented several novel temporal semi-discrete numerical scheme of the Cahn-Hilliard-Navier-Stokes equations using the time filtering technique, and have analyzed the unconditional stability and convergence of the backward Euler (BE) time filtering scheme. By incorporating the time filtering technique into the BE framework, we have shown that the computational complexity remains almost unchanged while improving the time order of the BE scheme by one. Through several numerical examples, we have demonstrated the effectiveness of the time-filtered scheme in enhancing both temporal accuracy, serving as a cheap estimator, and applying to some practical problems. Based on these results, we plan to extend to high-order algorithms characterized by conceptual simplicity and low computational cost.
	
	\section*{Data availability}
	Data will be made available on reasonable request.
	
	\section*{Declarations}
	The authors declare no competing interests.
	
	\appendix
	\section{Appendix A: Proof of Lemma \ref{Lemma_Estimation_Truncation_error}}\label{Appendix-Lemma-4.2}
	\begin{proof}
		For the first term, by the definition of $R_1,R_2,R_3$, we have
		\begin{equation*}
			\begin{aligned}
				\left\|R_1\right\|
				\leq & \left\|\frac{1}{\Delta t}\mathcal{A}\left(\phi(t^{n+1})\right)- \phi_t(t^{n+1})\right\| + \left\|\Delta\left(\mathcal{B}\left(\mu\left(t^{n+1}\right) \right)- \mu(t^{n+1})\right)\right\| \\
				& \quad + \left\|\left(\mathcal{B}\left(\boldsymbol{u}(t^{n+1})\right) - \boldsymbol{u}(t^{n+1})\right) \cdot \nabla \mathcal{B}\left(\phi(t^{n+1})\right)\right\| \\
				& \quad + \left\|\boldsymbol{u}(t^{n+1}) \cdot \nabla \left(\mathcal{B}\left(\phi(t^{n+1})\right) - \phi(t^{n+1})\right)\right\|  \\
				\leq & C\left\|\frac{1}{\Delta t}\mathcal{A}\left(\phi(t^{n+1})\right)- \phi_t(t^{n+1})\right\| + C\left\|\Delta\left(\mathcal{B}\left(\mu\left(t^{n+1}\right) \right)- \mu(t^{n+1})\right)\right\| \\
				& \quad + C\|\mathcal{B}\left(\boldsymbol{u}(t^{n+1})\right) - \boldsymbol{u}(t^{n+1})\|_1 \cdot \left\|\mathcal{B}\left(\phi(t^{n+1})\right)\right\|_2 \\
				& \quad + C\|\boldsymbol{u}(t^{n+1})\|_1 \cdot \left\|\mathcal{B}\left(\phi(t^{n+1})\right) - \phi(t^{n+1})\right\|_2
			\end{aligned}
		\end{equation*}
	    and 
	    \begin{equation*}
	    	\begin{aligned}
	    		\| \nabla R_2 \| &\leq \|\nabla \left(\mathcal{B}(\mu(t^{n+1}))-\mu(t^{n+1})\right)\| + \|\nabla \Delta \left(\mathcal{B}(\phi(t^{n+1}))-\phi(t^{n+1})\right)\| \\
	    		& \quad + \|\nabla \left(f(\phi(t^{n+1}))-2f(\phi(t^n))+f(\phi(t^{n-1}))\right)\|  \\
	    		& \leq  \|\nabla \left(\mathcal{B}(\mu(t^{n+1}))-\mu(t^{n+1})\right)\| + \|\nabla \Delta \left(\mathcal{B}(\phi(t^{n+1}))-\phi(t^{n+1})\right)\| \\
	    		& \quad + \|f^{\prime}(\xi_1)\nabla \left(\phi(t^{n+1})-2\phi(t^n)+\phi(t^{n-1})\right)\|  \\
	    		& \leq C \|\nabla \left(\mathcal{B}(\mu(t^{n+1}))-\mu(t^{n+1})\right)\| + C \|\nabla \Delta \left(\mathcal{B}(\phi(t^{n+1}))-\phi(t^{n+1})\right)\| \\
	    		& \quad + CL\left(\Delta t\right)^2 \|\phi_{tt}\|_{L^{\infty}(0,T;H^1)}.
	    	\end{aligned}
	    \end{equation*}
        Finally, 
	    \begin{equation*}
	    	\begin{aligned}
	    		\|R_3\| &\leq \left\|\frac{1}{\Delta t}\mathcal{A}(\boldsymbol{u}(t^{n+1}))-\boldsymbol{u}_t(t^{n+1})\right\| + \left\|\Delta\left(\mathcal{B}(\boldsymbol{u}(t^{n+1}))-\boldsymbol{u}(t^{n+1})\right)\right\|  \\
	    		& \quad + \left\| \mathcal{B}\left(\boldsymbol{u}(t^{n+1})\right) \cdot \nabla\left(\mathcal{B}\left(\boldsymbol{u}(t^{n+1})\right) - \boldsymbol{u}(t^{n+1})\right)\right\|  \\
	    		& \quad + \left\|\left(\mathcal{B}\left(\boldsymbol{u}(t^{n+1})\right) - \boldsymbol{u}(t^{n+1})\right) \cdot \nabla\boldsymbol{u}(t^{n+1})\right\|  \\
	    		& \quad + \left\| \mathcal{B}\left(\mu(t^{n+1})\right) \cdot \nabla \left(\mathcal{B}\left(\phi(t^{n+1})\right) - \phi(t^{n+1})\right) \right\|  \\
	    		& \quad + \left\|\left(\mathcal{B}\left(\mu(t^{n+1})\right) - \mu(t^{n+1})\right) \cdot \nabla\phi(t^{n+1})\right\|  \\
	    		& \quad +  \|\nabla\mathcal{B}(p(t^{n+1}))-p(t^{n+1}))\|  \\
	    		& \leq C\left\|\frac{1}{\Delta t}\mathcal{A}(\boldsymbol{u}(t^{n+1}))-\boldsymbol{u}_t(t^{n+1})\right\| + C\left\|\Delta\left(\mathcal{B}(\boldsymbol{u}(t^{n+1}))-\boldsymbol{u}(t^{n+1})\right)\right\|  \\
	    		& \quad + C\|\mathcal{B}\left(\boldsymbol{u}(t^{n+1})\right)\|_1 \cdot \left\|\mathcal{B}\left(\boldsymbol{u}(t^{n+1})\right) - \boldsymbol{u}(t^{n+1})\right\|_2  \\
	    		& \quad + C\|\boldsymbol{u}(t^{n+1})\|_2 \cdot \left\|\mathcal{B}\left(\boldsymbol{u}(t^{n+1})\right) - \boldsymbol{u}(t^{n+1})\right\|_1  \\
	    		& \quad + C\|\mathcal{B}\left(\mu(t^{n+1})\right)\|_1 \cdot \left\|\mathcal{B}\left(\phi(t^{n+1})\right) - \phi(t^{n+1}) \right\|_2  \\
	    		& \quad + C\|\phi(t^{n+1})\|_2 \cdot \left\|\mathcal{B}\left(\mu(t^{n+1})\right) - \mu(t^{n+1})\right\|_1  \\
	    		& \quad + C \|\nabla\mathcal{B}(p(t^{n+1}))-p(t^{n+1}))\|.
	    	\end{aligned}
	    \end{equation*}
	\end{proof}
    \section{Appendix B: Proof of Lemma \ref{Lemma_A_e_phi_dt}}\label{Appendix-Lemma-4.3}
    \begin{proof}
    	Setting $\psi=\mathcal{A}(e_{\phi}^{n+1})/\Delta t$ in \eqref{eq_error_phi} and using the definition of $\mathcal{T}$ to get
    	$$
    	\begin{aligned}
    		\left\|\frac{\mathcal{A}(e^{n+1}_{\phi})}{\Delta t}\right\|^2_{-1}  
    		= &- \left([\mathcal{B}\left(\boldsymbol{u}(t^{n+1})\right) \cdot \nabla \mathcal{B}\left(\phi(t^{n+1})\right) - \mathcal{B}(\boldsymbol{u}^{n+1}) \cdot \nabla \bar{\phi}^{n+1}], \mathcal{T}\left(\frac{\mathcal{A}(e^{n+1}_{\phi})}{\Delta t}\right)\right) \\
    		&- \left(\nabla \mathcal{B}(e^{n+1}_{\mu}), \nabla\mathcal{T}\left(\frac{\mathcal{A}(e^{n+1}_{\phi})}{\Delta t}\right)\right) + \left(R_1,\mathcal{T}\left(\frac{\mathcal{A}(e^{n+1}_{\phi})}{\Delta t}\right)\right).
    	\end{aligned}
    	$$
    	For the first term on the RHS, 
    	\begin{equation}
    		\begin{aligned}
    			& \left|\left([\mathcal{B}\left(\boldsymbol{u}(t^{n+1})\right) \cdot \nabla \mathcal{B}\left(\phi(t^{n+1})\right) - \mathcal{B}(\boldsymbol{u}^{n+1}) \cdot \nabla \bar{\phi}^{n+1}], \mathcal{T}\left(\frac{\mathcal{A}(e^{n+1}_{\phi})}{\Delta t}\right)\right)\right|  \\
    			= &~ \bigg|\left(\mathcal{B}(e_{\boldsymbol{u}}^{n+1})\nabla\mathcal{B}(\phi(t^{n+1}))+\mathcal{B}(\boldsymbol{u}^{n+1})\nabla\left(\mathcal{B}(\phi(t^{n+1}))-2\phi(t^{n})+\phi(t^{n-1})\right), \mathcal{B}(e_{\mu}^{n+1})\right)\\
    			&~+\left(\mathcal{B}(\boldsymbol{u}^{n+1})\nabla\left(2e_{\phi}^n-e_{\phi}^{n-1}\right),\mathcal{T}\left(\frac{\mathcal{A}(e^{n+1}_{\phi})}{\Delta t}\right)\right)\bigg|\\
    			\leq & \left(C\|\mathcal{B}(\phi(t^{n+1}))\|_2 \|\mathcal{B}(e_{\boldsymbol{u}}^{n+1})\| + C\|\nabla\mathcal{B}(\boldsymbol{u}^{n+1})\| \|\nabla\left(2e_{\phi}^n-e_{\phi}^{n-1}\right)\| \right. \\
    			& \quad \left.+ C\|\nabla\mathcal{B}(\boldsymbol{u}^{n+1})\| \|\nabla\left(\mathcal{B}(\phi(t^{n+1}))-2\phi(t^{n})+\phi(t^{n-1})\right)\|\right)\cdot \left\|\nabla \mathcal{T}\left(\frac{\mathcal{A}(e^{n+1}_{\phi})}{\Delta t}\right)\right\|  \\
    			\leq & \left(C\|\mathcal{B}(\phi(t^{n+1}))\|_2 \|\mathcal{B}(e_{\boldsymbol{u}}^{n+1})\| + C\|\nabla\mathcal{B}(\boldsymbol{u}^{n+1})\| \|\nabla\left(2e_{\phi}^n-e_{\phi}^{n-1}\right)\| \right. \\
    			& \quad \left.+ C\left(\Delta t\right)^2 \|\nabla \mathcal{B}(\boldsymbol{u}^{n+1})\| \cdot\|\phi_{tt}\|_{L^{\infty}(0,T;H^1)}\right)\cdot \left\|\frac{\mathcal{A}(e^{n+1}_{\phi})}{\Delta t}\right\|_{-1}
    		\end{aligned}
    	\end{equation}
    	For the second term on the RHS, 
    	\begin{equation}
    		\begin{aligned}
    			\left|\left(\nabla \mathcal{B}(e^{n+1}_{\mu}), \nabla\mathcal{T}\left(\frac{\mathcal{A}(e^{n+1}_{\phi})}{\Delta t}\right)\right)\right| 
    			\leq  \Bigg\|\nabla \mathcal{B}(e^{n+1}_{\mu})\Bigg\| \cdot \left\| \nabla \mathcal{T}\left(\frac{\mathcal{A}(e^{n+1}_{\phi})}{\Delta t}\right)\right\| 
    			= \Bigg\|\nabla \mathcal{B}(e^{n+1}_{\mu})\Bigg\| \cdot \left\| \frac{\mathcal{A}(e^{n+1}_{\phi})}{\Delta t}\right\|_{-1}
    		\end{aligned}
    	\end{equation}
    	For the third term on the RHS, 
    	\begin{equation}
    		\begin{aligned}
    			\left|\left(R_1,\mathcal{T}\left(\frac{\mathcal{A}(e^{n+1}_{\phi})}{\Delta t}\right)\right)\right|
    			\leq \Bigg\|R_1\Bigg\| \cdot \left\| \nabla \mathcal{T}\left(\frac{\mathcal{A}(e^{n+1}_{\phi})}{\Delta t}\right)\right\| 
    			\leq C \Bigg\|R_1\Bigg\| \cdot \left\| \frac{\mathcal{A}(e^{n+1}_{\phi})}{\Delta t}\right\|_{-1}
    		\end{aligned}
    	\end{equation}
    	Combining the above estimates to complete this proof.
    \end{proof}
	
	
	\bibliographystyle{plain}
	\bibliography{bibfile}

\begin{thebibliography}{10}

\bibitem{2025_Abbaszadeh_A_reduced_order_least_squares_support_vector_regression_and_isogeometric_collocation_method_to_simulate_Cahn_Hilliard_Navier_Stokes_equation}
Mostafa Abbaszadeh, Amirreza Khodadadian, Maryam Parvizi, Mehdi Dehghan, and
  Dunhui Xiao.
\newblock A reduced-order least squares-support vector regression and
  isogeometric collocation method to simulate
  {C}ahn-{H}illiard-{N}avier-{S}tokes equation.
\newblock {\em J. Comput. Phys.}, 523:Paper No. 113650, 23, 2025.

\bibitem{1972_FrequencyFilterforTimeIntegrations}
Richard Asselin.
\newblock Frequency filter for time integrations.
\newblock {\em Monthly Weather Review}, 100(6):487 -- 490, 1972.

\bibitem{2023_CaiWentao_Optimal_L2_error_estimates_of_unconditionally_stable_finite_element_schemes_for_the_Cahn_Hilliard_Navier_Stokes_system}
Wentao Cai, Weiwei Sun, Jilu Wang, and Zongze Yang.
\newblock Optimal {$L^2$} error estimates of unconditionally stable finite
  element schemes for the {C}ahn-{H}illiard-{N}avier-{S}tokes system.
\newblock {\em SIAM J. Numer. Anal.}, 61(3):1218--1245, 2023.

\bibitem{2017_CaiYongyong_Error_estimates_for_time_discretizations_of_Cahn_Hilliard_and_Allen_Cahn_phase_field_models_for_two_phase_incompressible_flows}
Yongyong Cai, Heejun Choi, and Jie Shen.
\newblock Error estimates for time discretizations of {C}ahn-{H}illiard and
  {A}llen-{C}ahn phase-field models for two-phase incompressible flows.
\newblock {\em Numer. Math.}, 137(2):417--449, 2017.

\bibitem{2024_CHEN_chuanjun_SCM}
Chuanjun Chen and Xiaofeng Yang.
\newblock Efficient fully-decoupled and fully-discrete explicit-{IEQ} numerical
  algorithm for the two-phase incompressible flow-coupled {C}ahn-{H}illiard
  phase-field model.
\newblock {\em Sci. China Math.}, 67(9):2171--2194, 2024.

\bibitem{2023_Chen_On_the_superconvergence_of_a_hybridizable_discontinuous_Galerkin_method_for_the_Cahn_Hilliard_equation_SINUM}
Gang Chen, Daozhi Han, John~R. Singler, and Yangwen Zhang.
\newblock On the superconvergence of a hybridizable discontinuous {G}alerkin
  method for the {C}ahn-{H}illiard equation.
\newblock {\em SIAM J. Numer. Anal.}, 61(1):83--109, 2023.

\bibitem{2022_ChenYaoyao_CHNS_2022_AMC}
Yaoyao Chen, Yunqing Huang, and Nianyu Yi.
\newblock Error analysis of a decoupled, linear and stable finite element
  method for {C}ahn-{H}illiard-{N}avier-{S}tokes equations.
\newblock {\em Appl. Math. Comput.}, 421:Paper No. 126928, 17, 2022.

\bibitem{2020_Aytekin_Cibik_Analysis_of_second_order_time_filtered_backward_Euler_method_for_MHD_equations}
Aytekin Cibik, Fatma Eroglu, and Song\"ul Kaya.
\newblock Analysis of second order time filtered backward {E}uler method for
  {MHD} equations.
\newblock {\em J. Sci. Comput.}, 82(2):Art. 38, 25, 2020.

\bibitem{2021_DeCaria_A_variable_stepsize_variable_order_family_of_low_complexity}
Victor DeCaria, Ahmet Guzel, William Layton, and Yi~Li.
\newblock A variable stepsize, variable order family of low complexity.
\newblock {\em SIAM J. Sci. Comput.}, 43(3):A2130--A2160, 2021.

\bibitem{2020_DeCaria_A_time_accurate_adaptive_discretization_for_fluid_flow_problems}
Victor DeCaria, William Layton, and Haiyun Zhao.
\newblock A time-accurate, adaptive discretization for fluid flow problems.
\newblock {\em Int. J. Numer. Anal. Model.}, 17(2):254--280, 2020.

\bibitem{2020_Layton_William_A_time_accurate_adaptive_discretization_for_fluid_flow_problems}
Victor DeCaria, William Layton, and Haiyun Zhao.
\newblock A time-accurate, adaptive discretization for fluid flow problems.
\newblock {\em Int. J. Numer. Anal. Model.}, 17(2):254--280, 2020.

\bibitem{2015_Diegel_Analysis_of_a_mixed_finite_element_method_for_a_Cahn_Hilliard_Darcy_Stokes_system}
Amanda Diegel, Xiaobing Feng, and Steven Wise.
\newblock Analysis of a mixed finite element method for a
  {C}ahn-{H}illiard-{D}arcy-{S}tokes system.
\newblock {\em SIAM J. Numer. Anal.}, 53(1):127--152, 2015.

\bibitem{2017_Diegel_Convergence_analysis_and_error_estimates_for_a_second_order_accurate_finite_element_method_for_the_Cahn_Hilliard_Navier_Stokes_system}
Amanda Diegel, Cheng Wang, Xiaoming Wang, and Steven Wise.
\newblock Convergence analysis and error estimates for a second order accurate
  finite element method for the {C}ahn-{H}illiard-{N}avier-{S}tokes system.
\newblock {\em Numer. Math.}, 137(3):495--534, 2017.

\bibitem{2025_Gao_CAMWA}
Haijun Gao, Xi~Li, and Minfu Feng.
\newblock A new decoupled unconditionally stable scheme and its optimal error
  analysis for the {C}ahn-{H}illiard-{N}avier-{S}tokes equations.
\newblock {\em Comput. Math. Appl.}, 194:53--85, 2025.

\bibitem{2025_Gao_NMTMA}
Haijun Gao, Xi~Li, and Minfu Feng.
\newblock Stability and error analysis of {SAV} semi-discrete scheme for
  {C}ahn-{H}illiard-{N}avier-{S}tokes model.
\newblock {\em Numer. Math. Theory Methods Appl.}, 18(1):66--102, 2025.

\bibitem{1990_Heywood_Finite_element_approximation_of_the_nonstationary_Navier_Stokes_problem_IV_Error_analysis_for_second_order_time_discretization}
John Heywood and Rolf Rannacher.
\newblock Finite-element approximation of the nonstationary {N}avier-{S}tokes
  problem. {IV}. {E}rror analysis for second-order time discretization.
\newblock {\em SIAM J. Numer. Anal.}, 27(2):353--384, 1990.

\bibitem{2025_FengXinlong_An_analysis_of_second_order_sav_filtered_time_stepping_finite_element_method_for_unsteady_natural_convection_problems}
Mengru Jiang, Jilian Wu, Ning Li, and Xinlong Feng.
\newblock An analysis of second-order {SAV}-filtered time-stepping finite
  element method for unsteady natural convection problems.
\newblock {\em Commun. Nonlinear Sci. Numer. Simul.}, 140:Paper No. 108365,
  2025.

\bibitem{2008_Kay_David_and_Welford_Richard_Finite_element_approximation_of_a_Cahn_Hilliard_Navier_Stokes_system}
David Kay, Vanessa Styles, and Richard Welford.
\newblock Finite element approximation of a {C}ahn-{H}illiard-{N}avier-{S}tokes
  system.
\newblock {\em Interfaces Free Bound.}, 10(1):15--43, 2008.

\bibitem{1971_ASEMIIMPLICITSCHEMEFORGRIDPOINTATMOSPHERICMODELSOFTHEPRIMITIVEEQUATIONS}
Michael Kwizak and Andre Robert.
\newblock A semi-implicit scheme for grid point atmospheric models of the
  primitive equations.
\newblock {\em Monthly Weather Review}, 99(1):32 -- 36, 1971.

\bibitem{2021_LiMinghui_New_efficient_time_stepping_schemes_for_the_Navier_Stokes_Cahn_Hilliard_equations}
Minghui Li and Chuanju Xu.
\newblock New efficient time-stepping schemes for the
  {N}avier-{S}tokes-{C}ahn-{H}illiard equations.
\newblock {\em Comput. \& Fluids}, 231:Paper No. 105174, 14, 2021.

\bibitem{2023_FengXinlong_Filtered_time_stepping_method_for_incompressible_Navier_Stokes_equations_with_variable_density}
Ning Li, Jilian Wu, and Xinlong Feng.
\newblock Filtered time-stepping method for incompressible {N}avier-{S}tokes
  equations with variable density.
\newblock {\em J. Comput. Phys.}, 473:Paper No. 111764, 24, 2023.

\bibitem{2022JieShen_LiXiaoliMSAVCHNStwo_phase_incompressible_flows}
Xiaoli Li and Jie Shen.
\newblock On fully decoupled {MSAV} schemes for the
  {C}ahn-{H}illiard-{N}avier-{S}tokes model of two-phase incompressible flows.
\newblock {\em Math. Models Methods Appl. Sci.}, 32(3):457--495, 2022.

\bibitem{2021_lixiaoli_New_SAV_NS}
Xiaoli Li, Jie Shen, and Zhengguang Liu.
\newblock New {SAV}-pressure correction methods for the {N}avier-{S}tokes
  equations: stability and error analysis.
\newblock {\em Math. Comp.}, 91(333):141--167, 2021.

\bibitem{2025_LiuXin_CHNSSDC}
Xin Liu, Dandan Xue, Shuaichao Pei, and Hong Yang.
\newblock {A second-order semi-implicit spectral deferred correction scheme for
  Cahn-Hilliard-Navier-Stokes equation}.
\newblock {\em Appl. Numer. Math.}, 216C:39--55, 2025.

\bibitem{2024_Ma_Error_analysis_with_polynomial_dependence_on_in_SAV_methods_for_the_Cahn_Hilliard_equation_JSC}
Shu Ma, Weifeng Qiu, and Xiaofeng Yang.
\newblock Error analysis with polynomial dependence on {$\varepsilon^{-1}$} in
  {SAV} methods for the {C}ahn-{H}illiard equation.
\newblock {\em J. Sci. Comput.}, 101(3):Paper No. 83, 24, 2024.

\bibitem{2023_QinYi_Analysis_of_a_new_adaptive_time_filter_algorithm_for_the_unsteady_Stokes_Darcy_model}
Yi~Qin, Yang Wang, Yi~Li, and Jian Li.
\newblock Analysis of a new adaptive time filter algorithm for the unsteady
  {S}tokes/{D}arcy model.
\newblock {\em Comput. \& Fluids}, 266:Paper No. 106055, 13, 2023.

\bibitem{2018_Shenjie_Xujie_CAEAFTSAVSYGF}
Jie Shen and Jie Xu.
\newblock Convergence and error analysis for the scalar auxiliary variable
  ({SAV}) schemes to gradient flows.
\newblock {\em SIAM J. Numer. Anal.}, 56(5):2895--2912, 2018.

\bibitem{2018_Shenjie_Xujie_SAV}
Jie Shen, Jie Xu, and Jiang Yang.
\newblock The scalar auxiliary variable ({SAV}) approach for gradient flows.
\newblock {\em J. Comput. Phys.}, 353:407--416, 2018.

\bibitem{2010_ShenJie_Energy_stable_schemes_for_Cahn_Hilliard_phase_field_model_of_two_phase_incompressible_flows}
Jie Shen and Xiaofeng Yang.
\newblock Energy stable schemes for {C}ahn-{H}illiard phase-field model of
  two-phase incompressible flows.
\newblock {\em Chinese Ann. Math. Ser. B}, 31(5):743--758, 2010.

\bibitem{2015_ShenJie_Decoupled_energy_stable_schemes_for_phase_field_models_of_two_phase_incompressible_flows}
Jie Shen and Xiaofeng Yang.
\newblock Decoupled, energy stable schemes for phase-field models of two-phase
  incompressible flows.
\newblock {\em SIAM J. Numer. Anal.}, 53(1):279--296, 2015.

\bibitem{2009_AProposedModificationtotheRobertAsselinTimeFilter}
Paul Williams.
\newblock A proposed modification to the {R}obert–{A}sselin time filter.
\newblock {\em Monthly Weather Review}, 137(8):2538 -- 2546, 2009.

\bibitem{2010_TheRAWFilterAnImprovementtotheRobertAsselinFilterinSemiImplicitIntegrations}
Paul Williams.
\newblock The {RAW} filter: An improvement to the {R}obert–{A}sselin filter
  in semi-implicit integrations.
\newblock {\em Monthly Weather Review}, 139(6):1996 -- 2007, 2011.

\bibitem{2023_FengXinlong_Analysis_of_a_filtered_time_stepping_finite_element_method_for_natural_convection_problems}
Jilian Wu, Ning Li, and Xinlong Feng.
\newblock Analysis of a filtered time-stepping finite element method for
  natural convection problems.
\newblock {\em SIAM J. Numer. Anal.}, 61(2):837--871, 2023.

\bibitem{2017_YangXiaofeng_Numerical_approximations_for_the_molecular_beam_epitaxial_growth_model_based_on_the_invariant_energy_quadratization_method}
Xiaofeng Yang, Jia Zhao, and Qi~Wang.
\newblock Numerical approximations for the molecular beam epitaxial growth
  model based on the invariant energy quadratization method.
\newblock {\em J. Comput. Phys.}, 333:104--127, 2017.

\end{thebibliography}
\end{document}